\documentclass[11pt]{amsart}

\usepackage{latexsym,amssymb,amsfonts,amsmath, amscd, euscript, MnSymbol}

\usepackage{amsthm}

\usepackage[running]{lineno} %to display line numbers
\def\wrt{{\em w.r.t.}}

\def\iff{if and only if}

\newtheorem{theorem}{Theorem}
\newtheorem{definition}{Definition}
\newtheorem{definitions}[definition]{Definitions}
\newtheorem{proposition}{Proposition}
\newtheorem{corollary}{Corollary}

\newtheorem{claim}{Claim}
\newtheorem{fact}{Fact}

\newtheorem{problem}{Problem}

\newtheorem{lemma}{Lemma}
\theoremstyle{definition}
\newtheorem{examples}{Examples}
\newtheorem{example}[examples]{Example}

\def\iff{if and only if}

\def\wrt{{\em w.r.t.}}

\usepackage{enumerate}

\newcommand{\R}{\mathbf{R}}
\newcommand{\RR}{\mathbb{R}}

\newcommand{\LL}{\mathbf{L}} %\L is already defined as pound
\newcommand{\N}{\mathbb{N}}

\newcommand{\Z}{\mathbb{Z}}
\newcommand{\Q}{\mathbb{Q}}

\newcommand{\M}{\mathbf{M}}

\DeclareMathOperator{\Spec}{Spec}
\DeclareMathOperator{\MSpec}{MSpec}
\DeclareMathOperator{\Nerv}{Nerv}
\DeclareMathOperator{\Iso}{Iso}
\DeclareMathOperator{\Ball}{Ball}
\DeclareMathOperator{\Path}{Path}
\DeclareMathOperator{\Sons}{Sons}
\DeclareMathOperator{\Fin}{Fin}
\DeclareMathOperator{\supp}{supp}
\DeclareMathOperator{\Well}{Well}
\DeclareMathOperator{\weight}{w}
\DeclareMathOperator{\density}{d}
\DeclareMathOperator{\piweight}{\pi w}
\DeclareMathOperator{\Max}{Max}
\DeclareMathOperator{\Min}{Min}
\DeclareMathOperator{\Sup}{Sup}
\DeclareMathOperator{\Inf}{Inf}
\DeclareMathOperator{\dom}{dom}
\DeclareMathOperator{\Past}{Past}

\title{On homogeneous  ultrametric spaces} 
\author[C.~Delhomm\'e]{Christian Delhomm\'e}
\address{ E.R.M.I.T.  D\'epartement de Math\'ematiques et
d'Informatique. Universit\'e de La R\'eunion, 15, avenue Ren\'e
Cassin, BP 71551,  97715 saint-Denis Messag. Cedex 9, La R\'eunion, France} \email {delhomme@univ-reunion.fr} 
\author[C.~Laflamme] {Claude Laflamme}
\thanks{The second author was supported by NSERC
of Canada Grant \# 690404}
\address{University of Calgary, Department of Mathematics and Statistics, Calgary, Alberta, Canada T2N 1N4} 
\email {laflamme@ucalgary.ca} 
\author [M.~Pouzet]{Maurice Pouzet}
\thanks{This research was completed  while the first and third author visited the Mathematical  Department  of the University  of Calgary in the fall 2007}
\address{Univ. Lyon, Universit\'e
Claude-Bernard Lyon1, CNRS UMR 5203, Institut Camille Jordan,   43 Bd. 11 Novembre 1918
F$69622$ Villeurbanne  cedex, France} \email{
pouzet@univ-lyon1.fr }
\author [N.~Sauer]{Norbert Sauer}
\thanks{The fourth author was supported by NSERC of
Canada Grant \# 691325} \address{University of Calgary, Department of
Mathematics and Statistics, Calgary, Alberta, Canada T2N 1N4} \email{
nsauer@ucalgary.ca}

\date{\today}
\begin{document}


\begin{abstract} 
A metric space $\M$ is \emph{homogeneous} if every isometry between finite subsets extends to a surjective isometry defined on the whole space. We show that if $\M$ is an ultrametric space, it suffices that isometries defined on singletons extend, \emph{i.e} that the group of isometries of $\M$ acts transitively. We derive this fact from a result expressing that the arity of the group of isometries of an ultrametric space is at most $2$. An illustration of this result with the notion of spectral homogeneity is given. With this, we show that the Cauchy completion of a homogeneous ultrametric space is homogeneous. We present several constructions of homogeneous ultrametric spaces, particularly the countable homogeneous ultrametric space, universal for rational distances, and its Cauchy completion. From a general embeddability result, we prove that every ultrametric space is embeddable into a homogeneous ultrametric space with the same set of distances values and we also derive three embeddability results due respectively to F.~Delon, A.~Lemin and  V.~Lemin,  and
V.~Fe\v{\i}nberg. Looking at ultrametric spaces as $2$-structures, we observe that the nerve of an ultrametric space is the tree of its robust modules, thus providing insight into possible further  structure results. 

\end{abstract}

\maketitle

\section*{Introduction and description of the results}
A relational structure $\R$ is \emph{homogeneous} if every isomorphism between finite induced substructures of $\R$ extends to an automorphism of the whole structure $\R$ itself. Introduced by Fra\"{\i}ss\'e \cite{fraisse} and J\'onsson \cite{jonsson}, homogeneous structures are now playing a fundamental role  in Model Theory. They occur in other parts of mathematics as well, e.g.  in the theory of groups and  in the theory of metric spaces. A prominent homogeneous metric space is the Urysohn space, a separable complete metric space, in  which every finite metric space is isometrically embeddable, and which is homogeneous in the sense that  every isometry between finite subsets extends to an isometry of the whole space onto itself.  Due to the  contributions of Pestov  \cite{pestov} and of Kechris, Pestov and Todorcevic \cite{KPT},  this space has  recently attracted some attention  in the theory of  infinite dimensional topological dynamics in connection with the study of extremely amenable groups. Subsequently,  some additional research on homogeneous metric spaces has developed. 

The study of indivisible metric spaces \cite{DLPS}, \cite {DLPS2}, \cite{DLPS3} led us to consider 
homogeneous ultrametric spaces.  We noticed in \cite{DLPS} that for  countable ultrametric spaces, the fact that the isometry group acts transitively ensures that the space is homogeneous. In this paper, we show that the countability condition is unnecessary (Corollary \ref{cor:imp}) \footnote{This result was announced in \cite{DLPS2}, see page 1465 line 6 with reference to a preliminary version of this paper. We derive this  fact from a result expressing that the arity of the group of isometries of an ultrametric space is at most $2$. The same conclusion for Polish ultrametric spaces was presented  by Maciej Malicki in his lecture to Toposym, Prague 2011, and published in \cite{malicki} (see 3) of Corollary 4.2 p. 1672); his proof is quite different.}. We  introduce  the notion of spectral homogeneity and give a characterization of spaces with that property (Theorem \ref{thm:char}).   From this characterization it  follows  that the Cauchy completion of a  homogeneous ultrametric space is homogeneous (Theorem \ref {thm:cauchy}). We describe several homogeneous ultrametric spaces, including the  countable homogeneous ultrametric spaces and their Cauchy completion. 

The spectrum $\Spec(\M)$ of an ultrametric space $\M$ is the set of values taken by the distance.  The \emph{degree} of a  closed ball $B$ of $\M$ is  the number  $\mathbf s_{\M}(B)$ of sons of $B$, that is the number of open balls within $B$ of the same radius as $B$ (see Definition \ref{def:sons}). The  \emph{degree sequence} of $\M$ is the cardinal function  $\mathbf {s}_{\M}$ which associates to each $r\in \Spec(\M)_*:=\Spec(\M)\setminus \{0\}$ the supremum $\mathbf {s}_{\M}(r)$ of $\mathbf s_{\M}(B)$ where $B$ has diameter $r$.  Given a subset $V$ of the non-negative reals which contains $0$, we look at the collection $\mathcal M_{\bf s} (V)$ of ultrametric spaces   whose spectrum is $V$ and cardinal function $\mathbf {s}$.  The class $\mathcal M_{\bf s} (V)$ contains   a strictly increasing sequence of length $\omega_1$ of homogeneous ultrametric spaces  provided that  $V$ contains the non-negative rational (Theorem \ref{aleph1chain}). Furthermore, every ultrametric  space $\M\in \mathcal M_{\bf s} (V)$ can be isometrically embedded into the space  $\M_{\mathbf s}(Well(V))$ made of functions $f\in \Pi_{r\in  V_{*}}  \mathbf s(r)$ whose support $\supp (f)$ is dually well ordered, the distance $d$ being defined by $d(f,g):= \Max \{r\in V: f(r)\not=g(r)\}$, and the function $\mathbf s$ equals to $\mathbf s_{\M}$ (Theorem \ref{mainresult}). Spaces of that form were characterized,  by   Fe\v{\i}nberg  \cite{feinberg1},  as homogeneous and $T$-complete ultrametric spaces (spaces in which every chain of non-empty balls has a non-empty intersection). With our embedding result, we give an other proof of  Fe\v{\i}nberg's result (with a slight improvement). We obtain the existence of  a space of density at most $\kappa^{\aleph_0}$ (namely $\M_{\mathbf \kappa}(Well(\RR_{*}^+)$)  in which  every ultrametric space $\M$ with  density at most $\kappa$  is isometrically embeddable (a result due to A. and V.  Lemin \cite{lemin2}). We show that the  subspace $\M_{\mathbf {s}}(\Fin(V))$ of $\M_{\mathbf s}(Well(V))$ made of functions with finite support can be isometrically embedded in every ultrametric spaces $\M$ such that  $\mathbf s(r) \leq \mathbf s_{\M}(B)$ for every closed ball $B$ with radius $r$ (a result essentially due to F.~Delon \cite{delon}). If  $V$ is countable and  $\mathbf {s}(r)$ is countable for each $r$  then it follows from Theorem 2.7 of \cite{DLPS} (as well as \cite {lionel}) that up to isometry $\M_{\mathbf {s}}(\Fin(V))$ is the unique countable homogeneous ultrametric space with degree sequence $\mathbf {s}$. Hence, countable homogeneous spaces are characterized by their spectrum and a cardinal function.  Since the Cauchy completion of $\M_{\mathbf {s}}(\Fin(V))$ is  also homogeneous, then provided that $\mathbf {s}$ is countable,  this Cauchy completion  is the unique  separable and Cauchy complete ultrametric space with degree sequence $\mathbf {s}$. But,  it is unlikely that there is a simple characterization of homogeneous ultrametric spaces of arbitrary cardinalities. Yet we conclude this paper by  establishing a link with the theory of $2$-structures; we observe that the nerve of an ultrametric space is the tree of its robust modules (Proposition \ref{prop:nerve}), thus providing a possible direction for further structural characterizations. 

We would also like to thank the anonymous referee whose numerous and
helpful comments improved the paper considerably.

\section{Basic notions on  ultrametric  spaces}

We recall the following notions. 
Let $\M:=(M,d)$ be a metric space, where $d$ is the distance function on $M$. Let $A$ be  a subset of $M$;  we denote by $d_{\restriction A}$ the restriction of $d$ to $A\times A$  and by $\M_{\restriction A}$ the metric space $(A, d_{\restriction A})$, which we call {\it the metric subspace of $\M$ induced on $A$}; the {\it diameter}
of  $A$ is $\delta(A): =\Sup\{d(x,y): x,y \in A\}$. If $x\in M$, the \emph  {distance} from $x$ to $A$ is $d(x, A):= \Inf\{d(x, y): y\in A\}$. 
Let  $a\in M$; for $r\in \RR^+$, the {\it open}, resp. {\it closed},  {\it ball of center $a$, radius $r$} is the set
$B(a, r):= \{x\in M: d(a,x)<r\}$, resp. $\hat{B}(a, r):=\{x\in M:  d(a,x)\leq r\}$. In the sequel, the term {\it ball} means an open or a closed ball. When needed, we denote by $\mathcal \Ball(\M)$ the collection of balls of $\M$. A ball is \emph{ non-trivial} if it has more than one element. Two balls, possibly in different metric spaces, have the  same \emph{kind} if they have the same diameter which is attained in both or in none. 
Let $\M:=(M, d)$ and $\M':=(M',d')$ be two metric spaces. A map $\varphi:M\rightarrow M'$ is an {\it isometry from $\M$ into $\M'$}, or an {\it embedding},   if 
\begin{equation}\label{eqisom}
d'(\varphi(x),\varphi(y))=d(x,y) \; \text{for all}\; x,y\in M
\end{equation} 
This is an isometry from $\M$ {\it onto} $\M'$ if it is surjective. In particular we denote by  
$\Iso(\M)$ the  group of surjective isometries of $\M$ onto $\M$. For brevity, and if this causes no confusion, we will say that a map from a subset $A$ of $M$ to a subset $A'$ of $M'$ is an isometry from $A$ to $A'$ if this  is an isometry from $\M_{\restriction A}$ onto  $\M'_{\restriction A'}$. We say that $\M$ is \emph{isometrically embeddable}   into $\M'$ if there is an isometry from $\M$ into $\M'$. If in addition there is no isometry from $\M'$ into $\M$ then $\M$ is  \emph{strictly isometrically embeddable} into $\M'$.

Four other notions will be of importance:

\begin{definitions} Let  $a\in M$, the \emph {spectrum}  of $a$ is the set 
\[ \Spec(\M,a):=\{d(a, x):  x\in M\}. \] 
The  \emph {multispectrum}  of     $\M$ is the set 
\[ \MSpec(\M):= \{ \Spec(\mathbf{M}, a): a\in M\}.\]
The \emph{spectrum} of $\M$ is the set 
\[ \Spec(\M):= \bigcup \MSpec(\M) \; (=\{ d(x,y): x,y\in M\}). \]
The \emph{nerve} of $\M$ is the set 
\[ \Nerv (\M):= \{\hat{B}(a,  r): a\in M, r\in \Spec(\M,a)\}. \] 
%The \emph{reduced nerve} of $\M$ denoted  $\check{\Nerv(\M)}$,  is obtained by taking out from $(\Nerv(\M),\supseteq)$ its singletons, that is the closed balls of radius $0$.
\end{definitions}

A metric space is {\em ultrametric} if it satisfies the strong triangle inequality $d(x,z)\leq \Max\{d(x,y),d(y,z)\}$. Note that a space is  ultrametric  if and only if $d(x,y)\geq d(y,z)\geq d(x,z)$ implies $d(x,y)=d(y,z)$. Alternatively, all triangles are isosceles, the two equal sides being the largest. In an ultrametric space, balls (open or closed, as defined above) are both open and closed with respect to the metric topology.   The essential property of ultrametric spaces, which follows trivially from the definition,  is that balls are either disjoint or comparable with respect to the subset relation. 
Observe also that the diameter $\delta(A)$ of a subset $A$ of an ultrametric space is equal, for any $a\in A$,
to $\sup\{d(a,x):x\in A\}$.
 It turns out that an ultrametric space can be recovered from the pair made of $(\Nerv(\M); \supseteq)$ and  the diameter function $\delta$. Ordered by the  reverse of the subset relation, the nerve of an ultrametric space is   a tree, more specifically a "ramified meet  tree" in which every element is  below a maximal element; the diameter function is a strictly decreasing  map from the tree in the non-negative reals, which is zero on the maximal elements of the tree. For the exact statement of this characterization, which we will not need here, see \cite{lemin} and also \cite{DLPS}. 

\begin{definition}\label{def:sons}
Let $\mathbf{M}:=(M,d)$ be an ultrametric space, $B\in \Nerv(\M)$ and $r:= \delta(B)$.  If $r>0$,  a \emph {son} of $B$ is  any open ball of radius $r$ which is a subset of $B$. We denote by $\Sons(B)$ the set of sons of $B$ and by $\mathbf{s_{\M}}(B)$ the cardinality of the set $\Sons(B)$. For every  $r\in \Spec(\M)\setminus \{0\}$ we denote by $\mathbf{s_{\M}}(r)$  the supremum of the cardinality  $\mathbf{s_{\M}}(B)$ where $B$ is any member of $\Nerv(\M)$ with diameter $r$. In the sequel, we identify $\mathbf {s}_{\M}(B)$ and $\mathbf {s}_{\M}(r)$ with ordinals  (more precisely with initial ordinals). 
 
\end{definition}
Note that the sons of $B$ form a partition of $B$. Also, note that they do not need to belong to $\Nerv(\M)$. These are the immediate successors in the tree of strong modules of $\M$ (see Section \ref{section:modules}) hence the terminology we use. 

\begin{examples}\label{examples1}
Let $\RR^{+}$ be the set of non negative reals. The map $d:  \RR^{+}\times \RR^{+} \rightarrow \RR^{+}$, defined by $d_{\RR^+}(x,y)= 0$ if $x=y$ and $d_{\RR^+}(x,y)= \Max \{x,y\}$ otherwise, is an ultrametric distance. It follows that every subset $V$ of $\RR^+$  containing  $0$ is the spectrum of some ultrametric space.

$\bullet$ Let $K:=2^\omega$ be the set of all $0-1$ valued $\omega$-sequences. Let $d: K\times K\to \mathbb{R^+}$ given by $d(x,y):= 0$ if $x=y$ and $d(x,y):=\frac{1}{\mu(x,y)+1}$ otherwise, where
$\mu(x,y):=\min\{n\in \omega: x(n)\not=y(n)\}$. It is easily checked that $\mathbf{K}:=(K,d)$ is an ultrametric space in which every non-trivial element of the nerve has two sons, which are again elements of the  nerve. Let $\mathbf {C}$ be the subspace of $\mathbf K$ containing all $\omega$-sequences with finite support, that is 1's appearing only finitely often. Members of the nerves from both spaces  having non zero radius identify to the binary tree of height $\omega$. Clearly, $\mathbf K$ is  the Cauchy completion  of $\mathbf {C}$.  

$\bullet$ Let $\kappa$ be a cardinal number at least equal to $2$, identified to the set of ordinal numbers of cardinality strictly less than $\kappa$ and let $L_\kappa$ be the set of all functions $f$ from $\Q^{+}_{*}$,  the positive rationals, to  $\kappa$ which are eventually zero, that is $f(x)=0$ for all $x\in \Q^{+}_{*}$ larger than some real $x(f)$. 
Let $d: \left[ L_{\kappa}\right]^2\to \mathbb{R_{+}}$ given by $d(f,g):= 0$ if $f=g$ and $d(f,g):=\Sup\{r\in \Q^{+}_{*}: f(r)\not=g(r)\}$ otherwise. 

It is easily checked that $\LL_{\kappa}:=(L_{\kappa},d)$ is an ultrametric space in which every non-trivial element of the nerve has continuum many sons and no son of an element of the nerve is an element of the nerve. Furthermore, as  shown by A. and V.~ Lemin \cite{lemin2}, every ultrametric space $\M$ of weight at most $\kappa$ can be  embedded isometrically in $\LL_{\kappa}$. 

It is easy to see that for each of these last two examples, the group of isometries acts transitively, hence from Corollary  \ref{cor:imp} below, these ultrametric spaces are homogeneous.  In Section \ref{section:transitivity} we will define more general constructions and also give an alternative proof of  Lemin's result.
\end{examples}
%Let d: L_{\kappa}^2\to \mathbb{R}$ given by $d(f,g):= 0$ if $f=g$ and $d(f,g):=\frac{1}{\mu(f,g)+1}$ for $f\not=g$ where $\mu(f,g):=\inf\{r\in \R: f(r)\not=g(r)\}$. It is easily checked that $\LL_{\kappa}:=(L_{\kappa},d)$ is an ultrametric space in which every non-trivial element of the nerve has continuum many sons and no son of an element of the nerve is an element of the nerve.}

The following facts for an ultrametric space $\M=:(M,d)$ and ball $B$ of $\M$  can easily be verified,  in the order they are stated. 

\begin{fact}\label{fact:stated}
\begin{enumerate}
\item If  $x,y\in B$ and $z\in M\setminus B$ then $d(x,z)=d(y,z)$.
\item If $\varphi$ is a function of $M$ to $M$ which induces  an isometry of $B$ into  $\varphi(B)$ and also an isometry of $M\setminus B$ onto $M\setminus \varphi(B)$, and if further this latter isometry can be  isometrically extended  to some element $x$, then $\varphi$ is  isometric.
\item If $B\in \Nerv(\M)$ and $x,y\in B$ so that the son of $B$ containing $x$ is different from the son of $B$ containing $y$ then $d(x,y)=\delta(B)$. If $x\in M$ and $r\in \Spec(\M,x)\setminus \{0\}$ then $\delta(\hat{B}(x,r))=r$ and $\Sons(\hat{B}(x,r))$ contains at least two different elements and the distance between elements in  different sons of $\hat{B}(x,r)$ is $r$.
\item Let $\psi$ be a bijection of $M$, $B,B^\prime\in \Nerv(\M)$ with $\delta(B)= \delta (B')$ and  $f$ be a bijection of $\Sons(B)$ to $\Sons(B^\prime)$. If  for every $A\in \Sons(B)$ the function  $\psi_{\restriction A}$ is an isometry of $A$ to $f(A)$ and $\psi_{\restriction M\setminus B}$ extends isometrically to some element $x$ of $B$, then $\psi$ is an isometry.
\item Let $B$ and $B^\prime$ be two balls in $\Nerv(\M)$ with $\delta(B)=\delta(B^\prime)$ and $x\in B$, $x^\prime\in B^\prime$.  If there exists an element $\varphi\in \Iso(\M)$ with $\varphi(x)=x^\prime$ then $\varphi[B]=B^\prime$ and $\varphi$ induces a bijection of  $\Sons(B)$ to $\Sons(B^\prime)$, and hence $\mathbf{s_{\M}}(B)=\mathbf{s_{\M}}(B^\prime)$. 
\end{enumerate}
\end{fact}

\section{Isometries and extensions}

% We say that  $\M$ {\it embeds into} $\M'$ if there is an embedding from $\M$ into $\M'$, that $\M$ and $\M'$ {\it equimorphic} if each embeds  into the other  
% and that $\M$ and  $\M'$ are {\it isometric} if there is an isometry from $\M$ onto $\M'$. 
%A \emph{local embedding from $\M$ into $\M'$} is any isometry from a subspace of $\M$ onto a subspace of $\M'$.  If $\M=\M'$, we will call it a local embedding of $\M$. 

Here is our first result:
\begin{theorem} \label{thm:first}Let $\M:=(M,d)$ be an ultrametric space, $n$ be a non negative integer and $F$, $F'$ be two $n$-element subsets of $M$. A map $\varphi: F\rightarrow F'$ extends to a surjective  isometry $\overline \varphi$ of 
$\M$ if  and only if:
\begin{enumerate}
\item $\varphi$ is an isometry from $\M_{\restriction F}$ onto $\M_{\restriction F'}$.
\item For each $x\in F$ the function $\varphi_{\restriction \{x\}}$ extends to a surjective  isometry of $\M$. 
\end{enumerate}
\end{theorem}
\begin{proof}
Trivially, Properties (1) and (2) follow from the existence of $\overline \varphi$. For the converse,  we argue by induction on $n$. For $n=0$, the identity map extends the empty one. For $n=1$, the hypothesis and the conclusion are the same. Suppose $n\geq 2$. Pick  $a\in F$ and set $r:= d(a, F\setminus \{a\})$ and   $B_0:=B(a, r)$ and  $B:=\hat{B}(a,r)$, and pick $b\in F\setminus\{a\}$ with $d(a,b)=r=\delta(B)$.

%Let $r:= \delta(A)$ and $r':=\delta(A')$. Since $f$ is an isometry, $r=r'$. set $B:= B'(A, r)$ ($:= \bigcup\{ B'(a,r): a\in A\}$) and similarly set $B:= B'(A', r)$.
%Proceed by induction of the size of the domain of the local automorphism (resp. spec-automorphism) $\varphi:F\to X$.
%If $F$ is empty then extend $\varphi$ as the identity.
%If it is a singleton, its is extendable by assumption.
%Assume that $F$ has at least two elements.
%Pick one $c\in F$.
By induction assumption, there are two  automorphisms of $\M$ denoted by  $\varphi_{\{a\}}$
and $\varphi_{F\setminus\{a\}}$ extending $\varphi_{\restriction\{a\}}$ and $\varphi_{\restriction F\setminus\{a\}}$ respectively. Let 
\begin{itemize}
\item $a'':= \varphi^{-1}_{F\setminus\{a\}}\circ \varphi_{\{a\}}(a)$,
\item $B_0':=\varphi_{\{a\}}[B_0]$, 
 \item $B_0'':=\varphi_{F\setminus\{a\}}^{-1}[B_0']$.
\end{itemize}
Because $\varphi_{\{a\}}(a)=\varphi(a)$ and $\varphi_{F\setminus\{a\}}(b)=\varphi(b)$ we get
\[
d(a'',b)= d(\varphi^{-1}_{F\setminus\{a\}}\circ\varphi(a),\varphi^{-1}_{F\setminus\{a\}}\circ\varphi(b))=d(\varphi(a),\varphi(b))=d(a,b)=r
\]
and hence from $d(a,b)=r$ and $d(a'',b)=r$, we have that  $d(a,a'')\leq r$ which in turn implies that $B_0$ and $B''_0$ are two sons of $B$ which are equal or disjoint. We  define  $\overline \varphi:M\rightarrow M$ via the following conditions: 
\begin{itemize}
 \item   $\varphi_{F\setminus\{a\}}$ on $M\setminus (B_0\cup B''_0)$,
 \item   $\varphi_{\{a\}}$ on $B_0$,
 \item  $\varphi _{F\setminus\{a\}}\circ \varphi_{\{a\}}^{-1}\circ \varphi_{F\setminus\{a\}}$ on $B_0''$ if $B_0''\not =B_0$.
\end{itemize}
If $B_0=B_0^{\prime\prime}$ then $\overline{\varphi}$ induces an isometry of  $M\setminus B_0$ and an isometry of $B_0$, hence is an isometry of $M$ onto $M$ according to Fact \ref{fact:stated} Item {\em (2)} with $a^{\prime\prime}$ used for  the element $x$ in Item $(2)$.

If $B_0\not= B_0^{\prime\prime}$ let $B':= \varphi_{F\setminus\{a\}}[B]$ and $f$ be the function of $\Sons(B)$ to $\Sons(B')$ with $f(A)=\varphi_{F\setminus\{a\}}[A])$ for every son $A\subseteq B\setminus(B_0\cup B_0^{\prime\prime})$ of $B$ and  $f(B_0)=B_0^\prime=\varphi_{\{a\}}[B_0]=\varphi_{F\setminus\{a\}}[B_0^{\prime\prime}]$ and $f(B_0^{\prime\prime})=\varphi_{F\setminus\{a\}}[B_0]$. Hence it follows from Fact \ref{fact:stated} Item {\em (4)}, again with $a^{\prime\prime}$ used for the element $x$, that $\overline{\varphi}$ is an isometry of $M$ onto $M$.

\end{proof}

\begin{corollary}\label{cor:1} Let $\M:=(M,d)$ be an ultrametric space and  $n$ be a non negative integer and $F$, $F'$ be two $n$-element subsets of $M$. A map $\varphi : F\rightarrow F'$ extends to a surjective  isometry $\overline \varphi$ of 
$\M$ if  and only if for each $x, y\in F$ the function  $\varphi_{\restriction \{x, y\}}$ extends to a surjective  isometry of $\M$. 
\end{corollary}
\begin{proof} The ``only if\thinspace" part is obvious; for the ``if\thinspace" part, observe that   $\varphi $ is an isometry of $\M_{\restriction F}$ onto  $\M_{\restriction F'}$.
\end{proof}

Corollary \ref{cor:1} describes  a property  of the action on $M$ of $\Iso(\M)$; in the terminology of \cite{cherlin}, it expresses that  $\Iso(\M)$  has arity at most $2$.

According to the terminology of Fra\"{\i}ss\'e \cite {Fra}, a metric space $\M$ is \emph{homogeneous} if every isometry  $f$  whose domain and range are finite subsets of $M$ extends  to surjective  isometry of $\M$ onto $\M$.

The following is a straightforward corollary to Theorem \ref{thm:first}.

\begin{corollary}\label{cor:imp} An ultrametric space $\M$ is homogeneous if and only if $\Iso(\M)$  acts transitively on $\M$.
\end{corollary}
We did not find mention of this  result in the literature.   We proved it  in \cite {DLPS} for countable ultrametric spaces, as  a step  in our characterization of  countable homogeneous  ultrametric spaces.

It is not true that every  isometry $\varphi: B\rightarrow A$ of a subset of a  homogeneous ultrametric space $A$ extends to an isometry defined on $A$. Indeed, if  $B$ is the image of $A$ by some isometry, its  inverse will not extend. One may ask whether   $\varphi$ extends if it  is an isometry of $B$  onto itself (the question was asked to J.~Melleray and communicated to us \cite{melleray}).  The answer is negative, even if $A$ is separable.  

\begin{example} Let $\omega^\omega$  be the set of integer valued $\omega$-sequence and  $\omega^{[\omega]}$ be the subset of those with finite support. Let $A$ be one of these  two sets equipped with the distance $d$ defined in Example \ref{examples1}. 
Set $A_{00}:= \{x\in A: x(0)=x(1)=0\}$, $A_{01^+}:= \{x\in A: x(0)=0, x(1)\geq 1\}$, $A_1:= \{x\in A: x(0)=1\}$, $A_{2^+}:= \{x\in A: x(0)\geq 2\}$. Let $B:=A_{01^+}\cup A_1 \cup A_{2^+}$  and $\varphi: B\rightarrow B$ which is the identity on  $ A_{2^+}$ and is the  bijection of order two from  $A_{01^+}$ onto  $A_1$ defined by $\varphi((0,i, \dots ))=(1, i-1,\dots))$. It is easy to check that $\varphi$ is an isometry from $B$ onto itself (note that the distance between any elements $x\in  A_{2^+}$, $y\in A_{01^+}$ and   $z\in A_1$ is $1$. But, if $t$ is an  element of $A_{00}$, there is no way to extend $\varphi $ to $t$. Indeed, pick $y\in A_{01^+}$. If $\varphi$ extends to $t$, then the image $t'$ of $t$ satisfies $d(t', \varphi(y))=d(t, y)=\frac{1}{2}$. Hence, $t'\in A_1$ and the extension of $\varphi$ is not one to one. \end{example} 

We conclude this section with an open problem. 

\begin{problem}
Corollary \ref{cor:imp} does not hold  for ordinary metric spaces. There are several metric spaces which are not homogeneous while their  automorphisms group is transitive ( a simple example is the product $\Z\times \Z$ equipped with the sup-distance ($d((x,y), (x',y'))= sup \{d(x,x'), d(y,y')\}$). The notion of ultrametric space extends to metric with values into a join-semilattice $V$ with a least element $0$. We ask for which $V$, ultrametric spaces with values in $V$ whose automorphism group is transitive are homogeneous?. \end{problem}

We may note that the answer is trivial if $V$ has two elements, and positive also if $V$ is the $4$-element Boolean algebra. 

%Note that the answer is trivial if $V$ has two elements. If $V:= \{0, a, \neg a, 1\}$ we may associate to a metric space $(E,d)$ the binary relational structure $R:= (E, \rho_a, \rho_{\neg a})$ where $\rho_a:= \{(x,y): d(x,y)\leq a$ and $\rho_{\neg a}$ is defined accordingly; these two relations are equivalence relations such that each equivalence class of one of these relations intersects each class of the other in at most an element.  Conversely, to a binary structure as above we associate the metric defined by $d(x,y):=0$ if $x=y$ and otherwise $d(x,y)= a$  if $(x,y)\in \rho_a$, $d(x,y)= \neg a$ if $(x,y)\in \rho_{\neg a}$, $d(x,y):= 1$ if $(x,y) \in E\times E \setminus \rho_{a} \cup \rho_{\neg a}$. The important feature of this correspondence is that non-expansive (or contracting) maps of $(E, d)$ are maps which preserve the two equivalences. Now, $Aut(R)$ is transitive if each class of $\rho_a$ intersects each class of $\rho_{neg a}$ in exactly one element. It is easy to see that such $R$ is homogeneous. Examples to consider next: 1) $V$ is the lattice $M_3$ : a three element antichain with a top and bottom added. Perhaps there we may find counterexamples (in this frame are the so called $3$-nets, presented in a paper with Christian). 2) $V$ is the $8$-element boolean algebra or any complete lattice which is completely meet-distributive. 

\section{Examples of homogeneous ultrametric spaces}

Corollary \ref{cor:imp} leaves us with the problem of deciding under which conditions $\Iso(\M)$  acts transitively on an ultrametric space  $\M$, that is to characterize the ultrametric spaces $\M:=(M,d)$  for which for all elements $x,x^\prime \in M$ there exists   $\varphi\in \Iso(\M)$ with $\varphi(x)=x^\prime$. In the following subsection we consider some necessary conditions. 

\subsection{Transitivity conditions}\label{section:transitivity}

\begin{definition}
An ultrametric  space $\M:=(M,d)$ has {\em property \textbf{h}} if it has properties $\mathbf{h_1}$ and $\mathbf{h_2}$ below, that is:
\begin{itemize}
\item[$\mathbf{h_1:}$] $\Spec(\M,x)=\Spec(\M,x^\prime)$ for all $x,x^\prime\in M$. 
\item[$\mathbf{h_2:}$] $\mathbf{s_{\M}}(B)=\mathbf{s_{\M}}(B^\prime)$, that is $|\Sons(B)|=|\Sons(B^\prime)|$,  for all $B,B^\prime\in \Nerv(\M)$ with $\delta(B)=\delta(B^\prime)$.
\end{itemize}
\end{definition}
Clearly, $\M$ has property $\mathbf{h_1}$ if and only if $\Spec(\M,x)=\Spec(\M)$ for every $x\in M$.  If $\M$ has property $\mathbf{h}$, we have $\mathbf{s_{\M}}(B)= \mathbf{s_{\M}}(r)$ for every   $B\in Nerv(\M)$ with diameter $r$. In this case, we call the map $\mathbf{s_{\M}}$ the \emph{degree sequence} of $\M$. 
If $\M$ is homogeneous  then,  trivially, it satisfies property $\mathbf{h_1}$. According to Item $(5)$ of Fact \ref{fact:stated} it satisfies property $\mathbf{h_2}$, hence: 

\begin{lemma}\label{lem:h}
Every homogeneous ultrametric space satisfies  property \textbf{h}. 
\end{lemma}

%Let $\M$ be an ultrametric space, $V:= \Spec (M)$ and $V_{*}:= V\setminus\{0\}$;  for each $B\in Nerv(\M)$,  let $\mathbf {s}(B):= \vert \Sons (B)\vert$. Denote by $\mathbf {s}_{\M}$ the map from $V_{*}$ into the cardinal numbers defined by $\mathbf {s}_{\M}(r):= \Sup \{\mathbf {s}_{\M}(B): B\in \Nerv(\M)\;  \text {and} \; \delta (B)= r\}$. 

Let $V$ be  a subset of $\R^+$ containing $\{0\}$. Let $V_{*}:= V\setminus \{0\}$. 
As we will see (in (3) of Proposition \ref{prop:ultrahom}), every map $\mathbf {s}$ which associate to each $r\in V_{*}$ some cardinal number $\mathbf {s}(r)\geq 2$ is the degree sequence $\mathbf {s}_{\M}$ of some homogeneous ultrametric space $\M$. Looking at the collection $\mathcal M_{\mathbf {s}}(V)$ of ultrametric spaces having spectrum $V$ and degree sequence $\mathbf {s}$, we note that if $V$ is countable and $\mathbf {s}(r)$ is countable for each $r$ then it follows from Theorem 2.7 of \cite{DLPS} that up to isometry there is a unique countable ultrametric  space $\M_{\mathbf {s}}(\Fin(V))$ with degree sequence $\mathbf {s}$ and  furthermore this ultrametric space  is homogeneous. Hence, for countable spaces, property \textbf{h} is a sufficient condition for transitivity. If one considers uncountable spaces then property \textbf{h} is not in general a sufficient condition for transitivity. Examples show that topological conditions have to be satisfied as well. Observing that the Cauchy completion of $\M_{\mathbf {s}}(\Fin(V))$ is  also homogeneous,  we show that, up to isometry, this is the unique  separable and Cauchy complete ultrametric space with degree sequence $\mathbf {s}$. But Cauchy completeness with property \textbf{h} is not in general a sufficient condition for transitivity.  It is unlikely that there be a  simple characterization of ultrametric spaces $\M$ for which $\Iso(\M)$ is transitive.

\begin{examples}\label{ex:counter}
\begin{enumerate}
\item
In the space $\mathbf K$ of Examples \ref{examples1},   let $\mathbf{0}$ be the constant 0-sequence $\mathbf{1}$ the constant 1-sequence, and {\mathversion{bold}$0^\prime$}  the sequence $(0,1,0,0,0,0,\ldots)$. Then $d(\boldsymbol{0},\boldsymbol{0}\boldsymbol{^\prime})=\frac{1}{2}$ and $d(\boldsymbol{0},\boldsymbol{1})=d(\boldsymbol{0}\boldsymbol{^\prime},\boldsymbol{1})=1$. Let $\mathbf K^\prime$ be obtained from $\mathbf K$ by removing the point $\boldsymbol{0}\boldsymbol{^\prime}$ from $\mathbf K$.  Then each member  of $\Nerv(\mathbf K^\prime)$ still has exactly two sons, and $\Spec(\mathbf K^\prime, x)=\Spec(\mathbf K^\prime, y)$ for all points $x,y$ of $\mathbf K^\prime$. Hence $\mathbf K^\prime$ satisfies Property \textbf{h}. But no element $\varphi$ of $\Iso(\mathbf K^\prime)$ maps $\mathbf{0}$ to $\mathbf{1}$ because $\varphi$ would have to map the ball $\hat{B}(\mathbf{0},\frac{1}{2})$ of $\mathbf K'$ onto the ball  $\hat{B}(\mathbf{1},\frac{1}{2})$ of $\mathbf K'$ which is impossible because the ball $\hat{B}(\mathbf{0},\frac{1}{2})$ is not Cauchy complete whereas the ball $\hat{B}(\mathbf{1},\frac{1}{2})$ is Cauchy complete. Note that this example also shows that requiring in addition to Property \textbf{h} that balls of the same diameter have the same cardinality still is not a sufficient condition. Of course a similar example can be constructed starting with the ultrametric space $\LL_{\kappa}$ of Examples \ref{examples1}.

\item The construction above will not work in the case of the countable ultrametric space $\mathbf {s}$ because no ball of $\mathbf {s}$ is Cauchy complete; of course it would also contradict  Theorem 2.7 of \cite{DLPS}.  But for a more direct argument, using the same elements  $\mathbf{0}, \boldsymbol{0}\boldsymbol{^\prime}$ and $\boldsymbol{1}$  for $\mathbf {s}$ as for $\mathbf K$ above,  we remove $  \boldsymbol{0}\boldsymbol{^\prime}$ from $\mathbf{C}$ to obtain the ultrametric space $\mathbb{C}^\prime$. Then we construct an isometry of the ball $\hat{B}(\mathbf{0},\frac{1}{2})$ of $\mathbf{C}^\prime$  onto the ball  $\hat{B}(\mathbf{1},\frac{1}{2})$ of $\mathbf{C}^\prime$ by a simple back and forth argument. The crucial step for this argument is to realize  that every partial isometry of  the ball   $\hat{B}(\mathbf{0},\frac{1}{2})$ of $\mathbf{C}^\prime$  onto the ball  $\hat{B}(\mathbf{1},\frac{1}{2})$ of $\mathbf{C}^\prime$ can be extended to include  any element of $\hat{B}(\mathbf{0},\frac{1}{2})$ or of $\hat{B}(\mathbf{1},\frac{1}{2})$. During this construction the missing limits of Cauchy sequences are simply rearranged. 

\item An example of Cauchy complete space satisfying property \textbf{h}  which is not homogeneous is the following. Let $V:=V_{0}\cup V_1\cup\{0\}$, where $V_i:=\{\frac{1}{n+2}+\frac{i}{2}: n\in \N\}$ for $i\in \{0,1\}$. Let $M:= 2^{(V_0\cup V_{1})}$ be the set of  $0-1$ maps $f$ defined on $V_0\cup V_1$ and $d: M\times M\to  V$ be given by $d(f,g):= 0$ if $f=g$ and $d(f,g):=\Max \{x\in V_{0}\cup V_1:  f(x)\not = g(x)\}$. The space $\M:=(M, d)$ is Cauchy complete and homogeneous. Its restriction $\M_{\restriction X}$ where $X:=\{f\in M: f(x)=0 \; \text {for some} \; x\in V_1\}$ is Cauchy complete but not homogeneous.

\end{enumerate}
\end{examples}

\subsection{Constructions of homogeneous ultrametric spaces} 

Let $V$ be a subset of $\RR^+$  which contains $0$ and let $V_{*}:= V\setminus \{0\}$.   Let $\mathbf {s}$ be a function which associate a cardinal number $\mathbf {s}(r)\geq 2$ to  each $r\in V_{*}$.  Let $V_{\mathbf {s}}:= \Pi_{r\in V_{*}}\mathbf {s}(r)$.  Viewing each $\mathbf {s}(r)$ as an ordinal number, let $0$ be its least element and let also  $0$ be the element  of $V_{\mathbf {s}}$ which takes value $0\in\mathbf {s}(r)$ for each $r\in  V_{*}$.  For $f,  g\in V_{\mathbf {s}}$, we set  $\Delta (f,g):=\{r\in V_{*}: f(r)\not =g(r)\}$. For $f\in V_{\mathbf {s}}$,   the \emph {support} of $f$ is the set  $\supp(f):=\Delta(f, 0)= \{r\in V_{*}: f(r)\not =0\}$.  Let $\sigma:=(\sigma_r)_{r\in V_{*}}$, where each  $\sigma_{r}$ is a permutation of $\mathbf {s}(r)$. 
 If $f\in V_{\mathbf {s}}$, set $\overline {\sigma}(f):= (\sigma_{r}(f(r)))_ {r\in V_{*}}$. This defines a permutation $\overline  \sigma$ of  $V_{\mathbf {s}}$. It satisfies:

 \begin{fact}\label{fact:inegalite}  $\Delta (\overline {\sigma}(f), \overline {\sigma}(g)) =\Delta (f, g)$ for every $f, g \in V_{\mathbf {s}}$. 
 \end{fact}

 \begin{fact}\label{fact:delta}$\Delta (f,g)\vartriangle \Delta (f,h)\subseteq \Delta (h,g)\subseteq \Delta (f,h)\cup \Delta (f,g)$ for every $f,g, h\in V_{\mathbf {s}}$.
\end{fact}

 Let  $\mathcal J$ be a collection of subsets of  $V_{*}$. Set $V_{\mathbf {s}}(\mathcal J):= \{f\in V_{\mathbf {s}}:  \supp(f)\in \mathcal J\}$. For $f,g\in V_{\mathbf {s}}$, we set $f\equiv_{\mathcal J}  g$ if $\Delta (f,g)\in \mathcal J$.
 
 \begin{lemma}  If  $\mathcal J$ is an initial segment of subsets of $V_{*}$ (that is closed under downward inclusion), then  $\mathcal J$ is an ideal of subsets of $V_{*}$ if and only if $\equiv_{\mathcal J}$ is an equivalence relation on $V_{\mathbf {s}}$. 
 \end{lemma} 

 \begin{proof}
 Suppose that  $\mathcal J$ is an ideal of subsets of $V_{*}$. Being non empty, it contains the empty set, hence $f\equiv_{\mathcal J} f$ for every $f$. Suppose that $f\ \equiv_{\mathcal J}g\equiv_{\mathcal J} h$. According to Fact \ref{fact:delta}, we have $ \Delta (f,h)\subseteq \Delta (f,g)\cup \Delta (h,g)$. Since $\mathcal J$ is an ideal, $ \Delta (f,h)\cup \Delta (h,g)\in \mathcal J$ and hence $\Delta (f,h)\in \mathcal J$, proving  $f\equiv_{\mathcal J}h$. The symmetry being obvious, it follows that $\equiv_J$ is an equivalence relation. 

Conversely, suppose that this is an equivalence relation.  Let $X, Y\in \mathcal J$, we show that $X\cup Y\in \mathcal J$. The set $X\cup Y$ is the support of its characteristic function $\chi_{X\cup Y}$ which belong to $V_{\mathbf {s}}$. Hence it suffices to show that $0\equiv_{\mathcal J} \chi_{X\cup Y}$.  The characteristic functions $\chi_{X\setminus Y}$ of  $X\setminus Y$ and  $\chi_{Y}$ of $Y$ belong to $V_{\mathbf {s}}$. Their supports are $X\setminus Y$ and $Y$ which belong to $\mathcal J$. Hence   $0\equiv_{\mathcal J} \chi_{X\setminus Y}\equiv_{\mathcal J} \chi_{X\cup Y}$. Thus  $0\equiv_{\mathcal J} \chi_{X\cup Y}$ as required.\end{proof}
 
For each subset $X$ of $\RR$ we set $\mu(X):= +\infty$ if $X$ is not bounded above in $\RR^+$ and $\mu (X):= \Sup_{\RR} (X)$ otherwise. Let $f, g\in  V_{\mathbf {s}}$. We set    $d(f,g):=\mu (\Delta(f,g))$. Let $Bound(V)$ be the collections of subsets $X$ of $V_{*}$ which are bounded above in $\RR^+$. Since $Bound(V)$ is an ideal, $\equiv_{Bound(V)}$ is an equivalence relation and  on each equivalence class, $d$  induces an ultrametric distance. If we extend the definition of ultrametric to functions taking infinite values then the pair  $\M_{\mathbf {s}}:=(V_{\mathbf {s}}, d)$ is an ultrametric space.  The group $\Iso(\M_{\mathbf {s}})$ of isometries of $\M_{\mathbf {s}}$ is transitive. Indeed, let  $f, g\in V_{\mathbf {s}}$, and   $\sigma:=(\sigma_r)_{r\in V\setminus \{0\}}$ where $\sigma_{r}$ is the transposition of $\mathbf {s}(r)$ which exchanges $f(r)$ and $g(r)$. Then  $\overline {\sigma}(f)=g$ and according to Fact \ref{fact:inegalite}, $\overline {\sigma}$ is an isometry of $\M_{\mathbf {s}}$. 
 Trivially, every isometry $u$  of $\M_{\mathbf {s}}$ preserves the equivalence relation  hence the image $C'$ of  an equivalence class $C$ is an equivalence class and thus $u$ is an isometry of $\M_{\mathbf {s} \restriction {C}}$ on $\M_{\mathbf {s} \restriction {C'}}$. Since $\Iso(\M_{\mathbf {s}})$ is transitive, the spaces induced on two different equivalence classes are isometric, and for each equivalence class $C$, $\Iso (\M_{\mathbf {s} \restriction {C}})$ is transitive. From Corollary \ref{cor:imp}, $\M_{\mathbf {s} \restriction {C}}$ is homogeneous. The equivalence class of $0$ is the set $V_{\mathbf {s}}(Bound(V))$ and the space induced on it is  $\M_{\mathbf {s}}(Bound(V))$.

 These facts extend to various ideals of subsets of $V_{*}$. 
 
  If $\mathcal J$ is a collection of subsets of $V_{*}$, we set  $\M_{\mathbf {s}}(\mathcal J):= \M_{\mathbf {s} \restriction V_{\mathbf {s}}(\mathcal J)}$.  We denote by $\Fin(V)$ the collection of finite subsets of $V_{*}$ and by $\Well(V)$ the collection of subsets of $V_{*}$ which are dually well ordered. If $\alpha$ is a countable ordinal, we denote by $\Well_{\alpha}(V)$ the subset of $\Well(V)$ made of subsets $A$ such that $\alpha$ is not embeddable into $(A, \geq)$. With this definition, $\Fin(V)= \Well_{\omega}(V)$. If $\alpha$ is an indecomposable ordinal (\emph{i.e.} $\alpha= \omega^{\beta}$ for some ordinal $\beta$), then $\Well_{\alpha}(V)$ is an ideal of subsets of $V_{*}$. 
 
 The following proposition provides several examples of homogeneous ultrametric spaces. 
 \begin{proposition} \label{prop:ultrahom}Let $\mathcal J$ be a subset of $Bound(V)$. Then:  \begin{enumerate}
 \item $\M_{\mathbf {s} }(\mathcal J)$ is an ultrametric space. \label {item:ultram}
\item \label{Item:Cauchy-completion}Let $\tilde {\mathcal J}:= \{A\in Bound(V): [a, +\infty [\cap A\in \mathcal J \;  \text{for all }\; a\in \RR^+\setminus \{0\}\}$. If $\mathcal J$ is an initial segment of $(\powerset(V_{*}), \subseteq)$ then
the space $\M_{\mathbf {s} }(\tilde{\mathcal J})$ is the Cauchy completion of $\M_{\mathbf {s} }({\mathcal J})$.
\item If $\mathcal J$  is an ideal of subsets of  $V_{*}$ then $\M_{\mathbf {s} }(\mathcal J)$ is a homogeneous ultrametric space and $\Spec(\M_{\mathbf {s} }(\mathcal J))={\mathcal J^{\vee}}:= \{Sup (W): W\in \mathcal J\}$.  If $
\Fin(V)\subseteq  \mathcal J\subseteq  \Well(V)$ then its degree sequence is $\mathbf{s}$. \label{item:spectrum} 
\end{enumerate}
 \end{proposition}
 
 \begin{proof}
Item \ref{item:ultram}. The space $\M_{\mathbf {s} }({\mathcal J})$ is the restriction to $V_{\mathbf {s}}(\mathcal J)$ of the ultrametric space $\M_{\mathbf {s} }(Bound(V))$. \\
Item \ref{Item:Cauchy-completion}. (a) $V_{\mathbf {s}}(\mathcal J)$ is topologically dense in $\M_{\mathbf {s}}(\tilde{\mathcal J})$. Let $f\in V_{\mathbf {s}}(\tilde{\mathcal J})$. Let $(a_n)_{n\in \N}$ be a strictly decreasing sequence of reals converging to zero. For each $n\in \N$, select $f_n\in V_{\mathbf {s}}$ such that $f_n(x)=f(x)$ for $x\in A_n:=[a_n, +\infty[\cap \supp(f)$ and $f_n(x)= 0$ for $x\in V\setminus (\supp(f)\cup \{0\})$.  Clearly, $\supp(f_n)= A_n\in \mathcal {J}$, thus $f_n\in  V_{\mathbf {s} }(\mathcal J)$. Clearly, $f=lim_{n\rightarrow +\infty}f_n$, hence $f$ belongs to the topological closure of  $V_{\mathbf {s} }(\mathcal J)$ in $\M_{\mathbf {s} }(\tilde{\mathcal J})$. (b) The space  $\M_{\mathbf {s}}(\tilde{\mathcal J})$ is Cauchy complete. Let $(f_n)_{n\in\N}$ be a Cauchy sequence in $\M_{\mathbf {s}}(\tilde{\mathcal J})$.  Let $(a_n)_{n\in \N}$ be a strictly decreasing sequence of reals converging to zero. Since $(f_n)_{n\in\N}$ is a Cauchy sequence, there is a strictly increasing sequence of integers $(n_k)_{k\in \N}$ such that the distance $d$ on $\M_{\mathbf {s}}(\tilde{\mathcal J})$ satisfies  $d(f_{n_k}, f_m)<a_k$ for all $k$ and $m$ such that $m\geq n_k$. Let $f$ be defined by $f(x):=f_{n_k}(x)$ for $x\in V\cap [a_k, a_{k-1}[$ for all $k\in \N$ (with the convention that $a_{-1}:=+\infty$). Then $\supp(f)\in \tilde{\mathcal J}$, that is $[a, +\infty [\cap \supp(f)\in \mathcal {J}$ for all $a\in \RR^+\setminus \{0\}$.  Indeed let $a\in \RR^+\setminus \{0\}$, and let $k\in \N$ such that   $a_k\leq a$. We further claim that $f$ and $f_{n_k}$ coincide on $[a_k, +\infty[$. For this, let $m\geq n_k$. For every $k'\leq k$ we have $d(f_{n_{k'}}, f_m)=\Sup (\Delta(f_{n_{k'}}, f_m))<a_{k'}$ meaning that $f_{n_{k'}} $ and  $f_m$ coincide on $[a_{k'}, +\infty[$. Due to its definition,   $f$ coincide with $f_m$ on  $[a_k, +\infty[$, proving our claim. Now, since $\supp(f_{n_k})\in \tilde{\mathcal J}$, $[a, +\infty[\cap \supp(f_{n_k})\in \mathcal J$. Hence, our claim ensures that $[a, +\infty [\cap \supp(f)\in \mathcal {J}$,  as required.

Item \ref{item:spectrum}. By Corollary \ref{cor:imp}, in order to prove that $\M_{\mathbf {s} }(\mathcal J)$ is homogeneous it suffices to prove that the group $\Iso(\M_{\mathbf {s} }(\mathcal J))$ acts transitively  on $\M_{\mathbf {s} }(\mathcal J)$. This verification is similar to that previously done in the case of $\mathcal J= Bound(V)$.

Now let $B\in \Nerv(\M_{\mathbf {s}})$ with $r=\delta(B)\not = 0$. For every  $x\in B$,  set $x_r$ for the restriction of $x$ to $V_{*}\cap ]r, \infty[$. Note that two elements  $x, x'$ of $B$ belong to two distinct sons if and only $x_r=x'_r$ and $x(r)\not = x'(r)$. Since $x(r),x'(r)\in \mathbf s(r)$, we have ${\mathbf s}_{\M_{\mathbf {s}}}(B)\leq \mathbf s(r)$. Let $y:=x_r$ for some $x\in B$. Since $\mathcal J$ is an initial segment, the element $\overline y$ of $V_{\mathbf{s}}$ which coincide with $y$ on   $V_{*}\cap ]r, \infty[$ and is $0$ everywhere else belongs to  $V_{\mathbf {s}}(\mathcal J)$. Since $\mathcal J$ is an ideal containing $\Fin (V)$, the elements of $V_{\mathbf{s}}$ which coincide with $\overline y$  outside $r$ and take any value belonging to ${\mathbf {s}(r)}$ are in $V_{\mathbf {s}}(\mathcal J)$. Hence ${\mathbf s}_{\M_{\mathbf {s}}}(B)= \mathbf s(r)$.  \end{proof}
 
%To describe the examples we will use the notation $\M\oplus \M^\prime$ for two ultrametric spaces $\M=(M,d)$ and $\M^\prime=(M^\prime,d^\prime)$ of diameter 1 and disjoint set of elements to denote the ultrametric space on $M\cup M^\prime$ for which the distance of an element in $M$ to an element in $M^\prime$ is 2 and the distances of points within $M$ and $M^\prime$ respectively are given by $d$ and $d^\prime$ respectively. 

\begin{lemma}\label{lem:almostsurj}If $V$ is dually well ordered and $\varphi$ is an isometry of  $\M:= \M_{\mathbf {s}}(Well(V))$ then there is some $f\in \M$ such that $\supp (\varphi(f))= V_*$. 
\end{lemma}

\begin{proof}Let $\gamma$ be the order type of $V_{*}$ dually ordered and $(r_{\alpha})_{\alpha<\gamma}$ be an enumeration of $V_{*}$ respecting the dual order. We define $f$ by induction. Let $\alpha < \gamma$. We suppose $f(r_{\beta})$ defined for every $\beta<\alpha$. Since $\mathbf s (r_{\alpha})\geq 2$,  the set  $B_{\alpha}$ of $g\in \M_{\mathbf {s}}(Well(V))$ such that $g(r_{\beta})=f(r_{\beta})$ for all $\beta<\alpha$ contains two elements $g_{\alpha, 0}$ and $g_{\alpha,  1}$ such that $g_{\alpha, 0}(r_{\alpha})\not = g_{\alpha,1}(r_{\alpha})$. Since $d_{\M}(g_{\alpha, 0}, g_{\alpha,  1})= r_{\alpha}$, $d_{\M}(\varphi(g_{\alpha, 0}), \varphi(g_{\alpha,  1}))= r_{\alpha}$, hence $\varphi(g_{\alpha, 0})(r_{\alpha})\not = \varphi(g_{\alpha, 1})(r_{\alpha})$. One of these two values is distinct from $0$. Take $f(r_{\alpha}):= g_{\alpha, i_{\alpha}}(r_{\alpha})$, where $i_{\alpha}$ is minimum in $\{0, 1\}$ such that $\varphi(g_{\alpha, i_{\alpha}})(r_{\alpha})\not =0$. To conclude, we check that $\supp (\varphi(f))= V_{*}$. For that, it suffices to observe  that  $\varphi (f)(r_{\alpha})= \varphi(g_{\alpha, i_{\alpha}})(r_{\alpha})$ for every $\alpha<\gamma$. This is obvious: $f$ and $g_{\alpha, i_{\alpha}}$ coincide up to $r_{\alpha}$; since $\varphi$ is an isometry, $\varphi(f)$ and $\varphi(g_{\alpha, i_{\alpha}})$ coincide up to $r_{\alpha}$.   \end{proof}

\begin{corollary}\label{cor:strictembedding}
Let $\alpha$ and $\beta$ be two countable ordinals with $\alpha<\beta$ and $V$ be  a subset of $\RR^+$ containing $\{0\}$  such that $V_{*}$ contains  a subset of type $\beta^*$, the dual of $\beta$. Then $\M_{\mathbf {s}}(Well_{\alpha}(V))$ is strictly isometrically embeddable into $\M_{\mathbf {s}}(Well_{\beta}(V))$. 
\end{corollary}

\begin{proof} Since $Well_{\alpha}(V) \subseteq Well_{\beta}(V)$, we conclude that $\M_{\mathbf {s}}(Well_{\alpha}(V))$  is an   isometric subset of  $\M_{\mathbf {s}}(Well_{\beta}(V))$. Suppose that there is an isometric embedding of $\M_{\mathbf {s}}(Well_{\beta}(V))$ into $\M_{\mathbf {s}}(Well_{\alpha}(V))$.
Let $V'$ be a subset of $V$ such that $V_{*}'$ has order type $\beta^*$. Let $\mathbf s':= \mathbf s_{\restriction V'}$. The subset of $\M_{\mathbf {s}}(Well_{\beta}(V))$ made of $f$ such that $\supp (f)\subseteq V'$ coincide with $Well (V')$, hence $\M_{\mathbf {s'}}(Well(V'))$ is an isometric subset of $\M_{\mathbf {s}}(Well_{\beta}(V))$ and thus isometrically embeds into  $\M_{\mathbf {s}}(Well_{\alpha}(V))$. This is impossible. Indeed, let $\varphi$ be an embedding. Let $\theta$ be defined by setting $\theta(f):= \varphi(f)_{\restriction V'}$ for every $f\in Well(V')$. Then $\theta$ is an isometric embedding of $\M_{\mathbf {s'}}(Well(V'))$ into $\M_{\mathbf {s'}}(Well_{\alpha}(V'))$ (indeed, let $f,g\in Well (V')$; since $\varphi$ is an isometry, $d(f,g)=d(\varphi (f), \varphi(g))$, since  $d(f,g)\in V'$, $d(\varphi (f), \varphi(g))=d(\varphi ( f)_{\restriction V'}, \varphi g_{\restriction V'})$). According to Lemma \ref{lem:almostsurj} there is some $f \in \M_{\mathbf {s'}}(Well(V'))$ such that $\supp (\theta (f))= V'_*$, contradicting the fact that $\theta(f)$ must belong to $Well_{\alpha} (V')$. 

\end{proof}

\begin{theorem}\label {aleph1chain} Let $V$ be a subset of $\RR^+$ containing $\Q^+$ and $\mathbf{s}$ be a cardinal function with domain $V_{*}$ and such that $\mathbf{s}(r)\geq 2$ for every $r\in V_{*}$. Then there is a strictly increasing  $\omega_{1}$-sequence of homogeneous ultrametric  spaces with spectrum $V$ and cardinal function $\bf{s}$. 
\end{theorem}

\begin{proof}
For each ordinal $\gamma$, let  $\M_{\gamma}:= \M_{\mathbf {s}}(Well_{\omega^{\gamma}}(V))$. Since $\omega^{\gamma}$ is indecomposable, $V_{\omega^{\gamma}}$ is an ideal of subsets of $V_{*}$. According to Proposition \ref {prop:ultrahom}, $\M_{\gamma}$ is homogeneous with spectrum $V$ and degree sequence ${\mathbb s}$. Clearly $\M_{\gamma}$ is embeddable into $\M_{\gamma'}$ whenever $\gamma<\gamma'$. Since $V$ contains $\Q^+$, we may apply  Corollary \ref{cor:strictembedding},  hence this embedding is strict. The sequence $(\M_{\gamma})_{\gamma<\omega_1}$ has the claimed property. \end{proof}

\subsection{Cauchy completion}Let $\M:=(M,d)$ be an ultrametric space, $X$ be a subset of $M$ and $x\in M$, we set $\Spec(\M_{\restriction X}, x):= \{d(x,y): y\in X\}$.  

\begin{lemma} \label{lem:dense}Let $\M:=(M,d)$ be an ultrametric space and $X$ be a topologically dense subset of $M$. Then:
\begin{enumerate}[{(1)}]
\item $\Spec(\M, x)= \Spec(\M_{\restriction X}, x)\cup \{0\}$ for all $x\in M$. In particular $\Spec(\M_{\restriction X}, x)=\Spec(\M, x)$ for all $x\in X$. 
\item $\Spec (\M_{\restriction X} )=\Spec(\M)$.
\item For every non trivial $B\in \Ball (\M)$,  $B\cap X\in \Ball(\M_{\restriction X})$, $B$ is the topological adherence in $\M$ of $B\cap X$ and has the same kind as $B\cap X$. 
\end{enumerate}
\end{lemma}

\begin{proof}
(1) Clearly, $\Spec(\M_{\restriction X}, x)\cup \{0\}\subseteq \Spec(\M, x)$. Conversely,   let $r\in \Spec(\M, x)$. If $r=0$, then there is nothing to prove. Otherwise, pick $y\in M$ such that $d(x,y)=r$. Since $X$ is dense, there is some $x'\in B(y,r)\cap X$. From the strong triangle  inequality, $d(x,x')=r$, hence $r\in  \Spec(\M_{\restriction X}, x)$ as claimed.\\   
\noindent (2) This assertion follows directly from (1). It is due to A.~Lemin \cite{lemin}. \\
\noindent (3) Let $B\in \Ball (\M)$ be a non trivial ball. Let $r:= \delta(B)$. Then, there is some $x\in M$ such that $B$ is either $B(x,r)$ or $\hat {B}(x,r)$. Since $X$ is dense and  $r\not =0$, we may pick some $x'\in B(x,r)\cap X$. Due to the strong triangle inequality, we have $B(x',r)=B(x,r)$  and $\hat B(x',r)=\hat B(x,r)$ from which  follows that $B\cap X\in \Ball(\M_{\restriction X})$. The fact that $B$ is the topological adherence of $B\cap X$ follows directly from the density of $X$ when $B=B(x,r)$.  When $B=\hat {B}(x,r)$ add the observation that by the strong triangle inequality $B(y,r')\subseteq B$ for every $y\in B$, $0<r'\leq r$. Since $B$ is the topological adherence of $B\cap X$, these two balls have  diameter $r$. By assertion (1), $r$ is attained in both or in none. Thus,  these balls have the same kind.
 \end{proof}

\begin{proposition} \label{prop:dense}Let $\M:=(M,d)$ be an ultrametric space and $X$ be a topologically dense subset of $M$. Then $\M_{\restriction  X}$ satisfies property $\mathbf{h_1}$, resp. property $\mathbf{h_2}$,  if and only if $\M$ does. Moreover, when property $\mathbf{h}$ is satisfied, the two spaces have the same degree sequence. 
\end{proposition}

\begin{proof} Suppose that $\M$ satisfies property $\mathbf{h_1}$.  Then  by (1) and (2) of Lemma \ref{lem:dense} we have $\Spec(\M_{\restriction X}, x)=\Spec(\M, x)=\Spec (\M_{\restriction X} )$ for every $x\in X$, amounting to the fact that $\M_{\restriction X}$ satisfies property $\mathbf{h_1}$.  For the converse, it suffices to prove  that $\Spec(\M, x)=\Spec(\M)$ for every $x\in M$. Let $x\in M$.  According to  (2) of Lemma \ref{lem:dense} it suffices to prove that $\Spec (\M_{\restriction X}) \subseteq \Spec(\M, x)$. Let $r\in \Spec (\M_{\restriction X})$. If $r=0$ there is nothing to prove. Suppose $r\not =0$. Since $X$ is dense, we may find $x'\in B(x,r)\cap X$. Since $\M_{\restriction X}$ satisfies property $\mathbf{h_1}$, $r\in \Spec(\M_{\restriction X}, x')$. Let $y\in X$ such that $d(x', y)=r$. Due to the strong triangle inequality, $d(x,y)=r$, hence $r\in \Spec(\M, x)$ as claimed. 

We prove now that $\M$ satisfies property $\mathbf{h_2}$ if and only if $\M_{\restriction X}$ satisfies property $\mathbf{h_2}$. From (3) of Lemma \ref{lem:dense} we deduce that   if  $B\in \Nerv(\M)$ then $B\cap X\in \Nerv(\M_{\restriction X})$ and $\mathbf{s_{\M}}(B)=\mathbf {s}_{\M_{\restriction X}}(B\cap X)$.  In particular,  if $\M_{\restriction X}$ has  property $\mathbf{h_2}$ then $\M$ too. For the converse, observe that if  $B'\in \Nerv(\M_{\restriction X})$, then for $r=\delta(B')$ and $x\in B'$, one has $B:= B(x, r)\in \Nerv(\M)$.  Apply then the previous argument.  \end{proof}

\begin{corollary} The Cauchy completion of an ultrametric space satisfying  property $\mathbf{h_1}$, resp. property $\mathbf{h}$, satisfies property $\mathbf{h_1}$, resp. property $\mathbf{h}$. \end{corollary}

We are now ready to characterize separable, Cauchy complete and homogeneous ultrametric spaces.

\begin{theorem}Let $\M:=(M,d)$ be an ultrametric space. The following properties are equivalent:
\begin{enumerate}[{(i)}]
\item $\M$ is separable, Cauchy complete and homogeneous.
\item $\M$ is isometric to the Cauchy completion of a space of the form $\M_{\mathbf {s}}(\Fin(V))$, where $V$ is a countable subset of $\R^+$ and $\mathbf {s}$ is a map from $V$ into $\omega+1\setminus 2$. 
\end{enumerate}
\end{theorem}
\begin{proof}
$(i)\Rightarrow (ii)$.  Let $X$ be a countable dense subset of $\M$. Then according to Proposition \ref{prop:dense}, $\M_{\restriction X}$ satisfies property $\mathbf{h}$. Let $V$ be its spectrum and $\mathbf {s}$ be its degree sequence. According to Theorem 2.7 of \cite{DLPS}, up to isometry  there is a unique countable ultrametric space  with degree sequence $\mathbf {s}$, namely  $\M_{\mathbf {s}}(\Fin(V))$. Hence $\M_{\restriction X}$ is isometric to $\M_{\mathbf {s}}(\Fin(V))$. Since $\M$ is Cauchy complete it is isometric to the Cauchy completion of  $\M_{\mathbf {s}}(\Fin(V))$. 

$(ii)\Rightarrow (i)$. Trivially $\M$ is separable and Cauchy complete.  Let $J:= \Fin(V)$. According to Proposition \ref{prop:ultrahom}, the Cauchy completion of $\M_{\mathbf {s}}(J)$ is $\M_{\mathbf {s}}(\tilde {J})$ and since $\tilde  {J} $ is an ideal, $\M_{\mathbf {s}}(\tilde {J})$ is homogeneous. This conclusion also follows from Theorem \ref{thm:cauchy} below. \end{proof}

\subsection{Embedding of an ultrametric space into a homogeneous one}

\begin{theorem}\label{mainresult}

Every ultrametric space $\M$ is embeddable into $\M_{\mathbf{s}}(Well(V))$ where $V=\Spec(\M)$ and $\mathbf {s}=\mathbf {s}_{\M}$. 

\end{theorem}
\begin{proof}The proof is in two steps. 
First step. We define an embedding $\varphi$ of $\M$ into an ultrametric space of the form $\M_{\mathbf{s'}}(Well(V))$ where $\mathbf {s}'$ defined on $V_{*}$ takes ordinal values. At the end of the process, it will appear that  $\mathbf {s}'(r)< \mathbf {s}(r)^{+}$ (that is $\mathbf {s}'(r)= \mathbf {s}(r)$ if $\mathbf {s}(r)$ is finite and $\mathbf {s}'(r)$ is an ordinal with the same cardinality as $\mathbf {s}(r)$). Let $\kappa:= \vert M\vert$ and $(a_{\alpha})_{\alpha<\kappa}$ be a transfinite enumeration of the elements of $M$. We define an embedding $\varphi$ from $\M$  into some $\M_{\mathbf{s'}}(Well(V))$.  The definition is  by induction on $\alpha$. We set $\varphi(a_{0})=0$, the constant function equal to $0$ everywhere. Let $\alpha >0$. We suppose that $\varphi$ is an isometry on the set $A_{\alpha}:= \{a_{\beta}: \beta< \alpha\}$. Our aim is to choose  $\varphi(a_{\alpha})\in \M_{\mathbf {s'}}(Well(V))$  in such a way that $\varphi$ becomes an isometry on $A_{\alpha+1}:= A_{\alpha} \cup \{a_{\alpha}\}$. For $r\in V$ we will define $\varphi(a_{\alpha})(r)$ depending how $r$ compares to $r_{\alpha}:= d(a_{\alpha}, A_{\alpha}) := \Inf \{d(a_{\alpha}, a_{\beta}): \beta<\alpha\}$.

$\bullet$  If  $r<r_{\alpha}$, we set $\varphi(a_{\alpha})=0$. 

$\bullet$ If $r>r_{\alpha}$, then $B_{\alpha}:=B(a_{\alpha}, r)\cap A_{\alpha}\not=\emptyset$. We set $\varphi(a_{\alpha})(r)=\varphi(a_{\beta})(r)$ where $a_{\beta}\in B_{\alpha}$. This value is independent of the choice of $a_{\beta}$. Indeed, for every $a_{\beta}, a_{\beta'}\in B_{\alpha}$, $\varphi(a_{\beta})$ and $\varphi(a_{\beta'})$ coincide on the interval $[r, \infty)\cap V^{*}$. Indeed, let $s:= d(a_{\beta}, a_{\beta'})$. Clearly $s<r$. Since $\varphi$ is an isometry on $A_{\alpha}$, we have $d(a_{\beta}, a_{\beta'})=d(\varphi(a_{\beta}), \varphi(a_{\beta'}))$. Hence $d(\varphi(a_{\beta}), \varphi(a_{\beta'}))=\Sup\{r':\in V_{*}: \varphi(a_{\beta})(r'), \varphi(a_{\beta'})(r')\}=s$. Thus, $\varphi(a_{\beta})$ and $\varphi(a_{\beta'})$ coincide on the interval $]s, \infty)\cap V^{*}$, and in particular on $[r, \infty)\cap V^{*}$. 

$\bullet$ Suppose that $r:=r_{\alpha}$. There are two cases.  Firstly, the distance from $a_{\alpha}$ to $A_{\alpha}$ is not attained. In this case, if $r_{\alpha}\not \in V_{*}$, $\varphi(a_{\alpha})$ is entirely defined already. If $r_{\alpha} \in V_{*}$, then we set  $\varphi(a_{\alpha})(r_{\alpha}):=0$. Secondly, the distance is attained. In this case, set  $C_{\alpha}:=  \{\varphi(a_\beta)(r_{\alpha}): a_{\beta}\in \hat B(a_{\alpha}, r_{\alpha})\cap A_{\alpha}\}$ and let $\varphi(a_{\alpha})(r_{\alpha}):= \Min (\mathbf {s}'(r) \setminus  C_{\alpha})$.  Note that if two elements $a_{\beta}$ , $a_{\beta'}$ belong to a son of $\hat B(a_{\alpha}, r_{\alpha})$ then $d(a_{\beta} , a_{\beta'})=d(\varphi(a_{\beta}), \varphi(a_{\beta'}))<r_{\alpha}$,  hence $\varphi(a_{\beta})(r_{\alpha})=\varphi(a_{\beta'})(r_{\alpha})$. This implies that $\vert C_{\alpha}\vert$ is at most the number of sons of $\hat B(a_{\alpha}, r_{\alpha})$ minus one (the son containing  $a_\alpha$ does not appear). 

We verify that $\varphi_{\alpha}$ has a dually well ordered support. By construction $\supp (\varphi(a_{\alpha}) )\subseteq [ r_{\alpha}, \infty)$ and for each $r\in V_{*}\cap ]r_{\alpha},   \infty) $, $\varphi (a_{\alpha})$ coincide with some $\varphi(a_{\beta})$ on $[r, \infty)$. Since from the induction hypothesis the support of $\varphi(a_{\beta})$ is dually well founded, the support of $\varphi(a_{\alpha})$ is dually well founded.  

Next we verify that $\varphi$ is an isometry on $A_{\alpha+1}$. Let $\beta< \alpha$ and $r:= d(a_{\alpha}, a_{\beta})$. Suppose that $r>r_{\alpha}$. Let $a_{\beta'}\in A_{\alpha}\cap B(a_{\alpha}, r)$. We have $d(a_{\beta}, a_{\beta'})=r$ and $d(a_{\alpha}, a_{\beta'})<r$. By construction,   $\varphi(a_{\alpha})$ coincides with   $\varphi(a_{\beta'})$  on $[r, \infty)$. Also, since the supports of our $\varphi(a_{\beta})'s$ are dually well ordered,  $r$ is the largest index such that $\varphi(a_{\beta})(r)\not = \varphi(a_{\beta '}) (r)$. Thus $\varphi(a_{\beta})(r)\not = \varphi(a_{\alpha})(r)$ hence  $d(\varphi(a_{\alpha}), \varphi(a_{\beta}))=r$. Suppose that $r=r_{\alpha}$.  By construction, $\varphi(a_{\alpha})$ coincides with   $\varphi(a_{\beta})$  on $]r, \infty)$ and $\varphi(a_{\alpha})(r_{\alpha})\not = \varphi(a_{\beta})(r_{\alpha})$. Hence $d(\varphi(a_{\alpha}), \varphi(a_{\beta}))=r$.

The map $\varphi$ is an isometry. A priori, the image is not in
$\M_{\mathbf {s}}(Well(V))$, but we only need a further modification.
To do this, given $\psi$ an embedding of an ultrametric space $\M$
into a space of the form $\M_{\mathbf {s}}(Well(V))$ and $B\in
\Nerv(\M)$ with $r:= \delta(B)\not =0$, we set $\tilde \psi(B):=
\{\psi(a)(r): a \in B\}$.

\begin{claim}\label{claim:initial} For every $B\in \Nerv(\M)$ with $\delta(B)=r\not =0$, the set  $\tilde\varphi(B)$ is an  initial segment of $\mathbf {s}'(r)$ with cardinality $\mathbf {s}_{\M}(B)$. Furthermore, if  $\gamma\in \tilde\varphi(B)$ and  $\alpha$ is the least ordinal such that $a_{\alpha}\in B$ and $\varphi(a_{\alpha})(r)=\gamma$ then $\varphi(a_{\alpha})(r')=0$ for every $r'<r$ with $r'\in V_{*}$. 
\end{claim} 

\noindent{\bf Proof of Claim \ref{claim:initial}.}
The fact that $\vert \tilde \varphi(B)\vert=\mathbf {s}_{\M}(B)$ is immediate. Indeed,    $a_{\alpha}$ and $a_{\alpha'}$ belong to the same son of $B$ iff  $d(a_{\alpha}, a_{\alpha'})<r$ which is equivalent to $\varphi(a_{\alpha})(r')=\varphi(a_{\alpha'})(r')$ for every $r'\geq r$. Since $d(a_{\alpha}, a_{\alpha'})\leq r$, $d(a_{\alpha}, a_{\alpha'})<r$ is equivalent to $\varphi(a_{\alpha})(r')=\varphi(a_{\alpha'})(r')$. We verify that  $\tilde \varphi(B)$ is an  initial segment of $\mathbf {s}'(r)$.  First, $0\in \tilde \varphi(B)$. Indeed, let $\alpha$ be minimum such that $a_{\alpha}\in B$. If $\alpha=0$ then since $\varphi(a_{0})=0$ we have $\varphi(a_{0})(0)=0$ and $0\in \tilde \varphi(B)$ as required. If $\alpha\not =0$ then $d(A_{\alpha}, a_{\alpha}) >r$ indeed, if $d(a_{\beta}, a_{\alpha})\leq r $ for some $\beta< \alpha$ then $a_{\beta}\in B$ contradicting the minimality of $\alpha$. Since $d(A_{\alpha}, a_{\alpha}) >r$ we have by construction $\varphi(a_{\alpha})(r)=0$. Hence $0\in \varphi(\tilde B)$.  Now, let $\gamma\in \tilde \varphi(B)$. We may suppose that for every $\gamma ' \in \tilde \varphi(B)$  with $\gamma'<\gamma$, the interval $(\leftarrow \gamma'] $ is included into  $\tilde \varphi(B)$. Let $\alpha$ be  minimum  such that $a_{\alpha}\in B$ and $\varphi(a_{\alpha})(r)= \gamma$. Then $r=r_{\alpha}$. This is because by construction,  $\varphi(a_{\alpha})(r')=0$ for $r'<r_{\alpha}$, we have $r_{\alpha}\leq r$. If $r_{\alpha}< r$ then $\varphi(a_{\alpha})(r)= \varphi(a_{\beta})(r)$ for some $\beta<\alpha$,  contradicting our choice of $\alpha$. Since $\varphi(a_{\alpha})(r)= \gamma\not =0$, the distance $r_{\alpha}$ is attained. According to our definition $\varphi(a_{\alpha})(r_{\alpha}) = \Min (\mathbf {s}'(r) \setminus  C_{\alpha})$, where   $C_{\alpha}:=  \{\varphi(a_\beta)(r_{\alpha}): a_{\beta}\in \hat B(a_{\alpha}, r_{\alpha})\cap A_{\alpha}\}$. According to our hypothesis $(\leftarrow \varphi(a_\beta)(r_{\alpha})]\subseteq \tilde \varphi(B)$ for every $\beta$, hence $C_{\alpha}$ is an initial segment of $\mathbf {s}'(r)$. Since $\gamma$ is the least element of $\mathbf {s}'(r) \setminus  C_{\alpha}$, $(\leftarrow \gamma]\subseteq \tilde \varphi(B)$. It follows that $\tilde \varphi(B)$ is an initial segment of $\mathbf {s}'(r)$ and that the second part of our claim is satisfied.  
\hfill $\Box$

For each  $B\in \Nerv(\M)$, with $\delta(B)\not = 0$   let $\mathbf {s}_{B}$ be  a bijective map from $\tilde \varphi(B)$ onto $\mathbf{s}_{\M}(B)$ viewed as an ordinal and  such that $0$ is mapped on $0$. 

For each $\alpha<\kappa$, and $r\in V_{*}$, set $\psi (a_{\alpha})(r):=0$ if no member of $\Nerv(\M)$ containing $a_{\alpha}$ has diameter $r$, otherwise set $\psi(a_{\alpha})(r):=\mathbf {s}_{B} (\varphi(a_{\alpha})(r))$ where $B$ is the unique member of $\Nerv(\M)$ which contains $a_{\alpha}$ and has diameter $r$.

Since for each $r\in V_{*}$,  $\psi(a_{\alpha})(r)\in\mathbf {s}_{B}$,  $\psi(a_{\alpha})\in V_{\mathbf {s}}(Well(V))$. This allows to define  a map $\psi$ from $\M$ into  $\M_{\mathbf {s}}(Well(V))$. 

\begin{claim}\label{claim:secondstep} The map $\psi$ is an isometry of $\M$ into  $\M_{\mathbf {s}}(Well(V))$. 
\end{claim}

\noindent{\bf Proof of Claim \ref{claim:secondstep}.}
Let $\beta <\alpha$ and $r:= d(a_{\alpha}, a_{\beta})$. Let $B:= \hat B(a_{\alpha}, r)$. Since $a_{\alpha}, a_{\beta}\in \M$, $B\in \Nerv(\M)$, hence $\psi(a_{\alpha})(r)=\mathbf {s}_{B} (\varphi(a_{\alpha})(r))$ and $\psi(a_{\beta})(r)=\mathbf {s}_{B} (\varphi(a_{\beta})(r))$. Since $\varphi$ is an isometry, $\varphi(a_{\alpha})(r)\not= \varphi(a_{\beta})(r)$ and since $\mathbf {s}_B$ is one-to-one, $\psi(a_{\alpha})(r)\not = \psi(a_{\beta})(r)$. Let $r'>r$ and $B\in \Nerv(\M)$ with $\delta(B)=r'$. If $B$ contains $a_{\alpha}$ it contains $a_{\beta}$ and conversely. If no such $B$ contains $a_{\alpha}$ then $\psi(a_{\alpha} )=0$ and similarly $\psi(a_{\beta} )=0$. Suppose some $B$ contains $a_{\alpha}$ and $a_{\beta}$. We have $\psi(a_{\alpha})(r'):=\mathbf {s}_{B} (\varphi(a_{\alpha})(r'))$ and $\psi(a_{\beta})(r'):=\mathbf {s}_{B} (\varphi(a_{\alpha})(r'))$. Since $r'>r$  $\varphi(a_{\alpha})(r')=\varphi(a_{\beta})(r')$. Hence  $\psi(a_{\alpha})(r')= \psi(a_{\beta})(r')$. It follows that $d(\psi(a_{\alpha}), \psi(a_{\beta}))=r$. Hence $\psi$ is an isometry.\hfill $\Box$

With this claim, the proof of Theorem \ref{mainresult}  is complete. \hfill $\Box$

The embedding $\psi$ has an extra property that we will use in the proof of Theorem \ref{thm:feinberg} below. We  give it now. 

\begin{claim}\label{claim:extrastep} Let  $B\in \Nerv(\M)$ with $r:=\delta(B)\not =0$. Then $\tilde \psi (B)= \mathbf {s}_{\M}(B)$. \\
Further, if $\zeta\in \mathbf {s}_{\M}(B)$ and  $\alpha$ is the least ordinal such that $a_{\alpha}\in B$ and $\psi(a_{\alpha})(r)=\zeta$, then $\psi(a_{\alpha})(r')=0$ for every $r'<r$ with $r'\in V_{*}$. 
\end{claim}

\noindent{\bf Proof of Claim \ref{claim:extrastep}.}
Let $B\in \Nerv(\M)$ and $r:=\delta(B)\not =0$.  By definition, $\tilde \psi (B)= \{\psi(x)(r): x\in B\}$, and hence $\tilde \psi (B)=\mathbf s_{B}(\tilde \varphi(B))$. Since  $\mathbf s_{B}$ is a bijective map of $\tilde \varphi(B)$ onto $\mathbf{s}_{\M}(B)$, we have $\tilde \psi (B)= \mathbf {s}_{\M}(B)$,  as claimed. 

Now let  $\zeta\in \mathbf {s}_{\M}(B)$ and  $\alpha$ is the least ordinal such that $a_{\alpha}\in B$ and $\psi(a_{\alpha})(r)=\zeta$. Then choose $\gamma\in B$ such that $\psi(\gamma)=\zeta$, and thus $\alpha$ is minimum such that  $\varphi(a_{\alpha})(r)=\gamma$; hence Claim \ref{claim:initial} applies and we have $\varphi(a_{\alpha})(r')=0$ for every $r'<r$ with $r'\in V_{*}$. Since $\mathbf {s}_{B}$ maps $0$ to $0$,  $\psi(a_{\alpha})(r') =\mathbf {s}_{B}(\varphi(a_{\alpha}(r'))= 0$ for every $r'<r$ with $r'\in V_{*}$. 
\end{proof}

Since $\M_{\mathbf {s}}(Well(V))$ is homogeneous with spectrum $V$ we obtain:

\begin{theorem}\label{ultram-homog-embed}
Every ultrametric space embeds isometrically into a homogeneous ultrametric space with the same spectrum. 
\end{theorem}

An ultrametric space is \emph{$T$-complete} if the intersection of every chain of non-empty balls is non-empty. 
\begin{fact} An ultrametric is $T$-complete iff the intersection of every chain of members of the nerve is non-empty. 
\end{fact}
For that it suffices to observe that if $B$  and $B'$ are two  non-empty balls with $B'\subset B$, there is a member $D\in \Nerv(\M)$ such that $B'\subset D \subseteq B$ (namely $D:= \hat B(x, d(x,y))$ where $x\in B'$  and $y\in B\setminus B'$).

\begin{theorem}\label{thm:feinberg}The following properties are equivalent for an  ultrametric space $\M$:
\begin{enumerate} [{(i)}]
\item $\M$ is $T$-complete and homogeneous;
\item $\M$ is $T$-complete and satisfies condition $\bf h$; 
\item $\M$ is isometric to some $\M_{\mathbf {s}}(Well(V))$. 
\end{enumerate}
\end{theorem}
\begin{proof} 
$(i)\Rightarrow (ii)$. This part follows from  Lemma \ref{lem:h}. \\
$(iii)\Rightarrow (i)$. This follows by applying Proposition \ref{prop:ultrahom}. \\
$(ii)\Rightarrow (iii)$. Let $\kappa:= \vert M\vert$, 
$(a_{\alpha})_{\alpha<\kappa}$ be a transfinite enumeration of the elements of $M$ and  $\psi: \M\rightarrow \M_{\mathbf {s}}(Well(V))$ given by Theorem \ref{mainresult}. We prove that $\psi$ is surjective. We prove that for each ordinal $\alpha$, every $f\in  V_{\mathbf {s}}(Well(V))$ whose type of  $A:= \supp (f)$ equipped with the dual order of $V_{*}$ is $\alpha$ there is some $x\in \M$ such that $\psi (x)=f$. The proof is by induction on $\alpha$. We consider  two cases. 

\noindent{\bf Case 1.}  $\alpha$ is a limit ordinal. If $\alpha=0$  then $f=0$. Hence,  $a_0$ will do. Suppose  $\alpha >0$. Since $\alpha$ is denumerable  and limit, there is some strictly descending sequence $(r_n)_{n<\omega}$   of elements of $A$ which is coinitial, that is, every element $r\in A$ dominates some $r_n$. For each  non-negative integer $n$ the  type of the opposite order on   $A_n:= ]r_n, \infty[ \cap A$ is strictly smaller than $\alpha$. The map  $f_n$ which coincides with $f$ on $A_n$ and takes the value $0$ elsewhere belongs to $\M_{\mathbf {s}}(Well(V))$. Hence, according to the induction hypothesis, there is some $x_n\in \M$ such that $\psi (x_n)= f_n$.  Let $n< m<\omega$ then $d(x_n,x_m)= d(f_n,f_m)=d(f_n,f)= r_n$. Hence, the sequence of balls $B_n:= \hat B_{\M}(x_n, r_n)$ is decreasing. Since $\M$ is $T$-complete, $B_{\omega}:= \bigcap_{n<\omega}B_n$ is non-empty. Let $\alpha$ be the least ordinal such that $a_{\alpha}\in B_{\omega}$. We claim that $\psi (a_{\alpha})=f$, that is $\psi (a_{\alpha})(r)=f(r)$ for every $r\in V_{*}$.  Let $r\in V_{*}$. Subcase 1.  $r\leq r_n$ for all $n<\omega$. Then $B:= \hat B(a_{\alpha}, r)\subseteq  B_{\omega}$ (indeed, let $n<\omega$. Since $a_{\alpha}\in B_n=\hat B_{\M}(x_n, r_n)$ and $r\leq r_n$, we have $\hat B(a_{\alpha}, r)\subseteq  B_n$). Since $\M$ satisfies condition $\bf h_1$, $\hat B(a_{\alpha}, r)\in \Nerv(\M)$. Due to our choice of $\alpha$, this is the least ordinal such that $a_{\alpha}\in B$. According to the proof of Claim \ref{claim:initial} in the proof of Theorem \ref{mainresult}, $\varphi(a_{\alpha})(r)=0$, hence $\psi (a_{\alpha})(r)=0$ and since $f(r)=0$, $\psi (a_{\alpha})(r)=f(r)$.  Subcase 2. $r_n<r$ for some $n$. Since $a_{\alpha}\in B_n$, we have $d(a_{\alpha}, x_n)\leq r_n<r$. Since $\psi$ is an isometry, $d(\psi(a_{\alpha}), \psi(x_n))<r$ hence $\psi(a_{\alpha})(r)= \psi(x_n)(r)$. Since $\psi(x_n)=f_n$ and $f_n(r)=f(r)$, we have $\psi (a_{\alpha})(r)=f(r)$. This proves our claim. 

\noindent{\bf Case 2.} $\alpha$ is a successor ordinal: $\alpha= \alpha'+1$. Let $r$ be the least element of $A$ and   $A':= A\setminus \{r\}$. The map  $f'$ which coincides with $f$ on $A'$ and takes the value $0$ elsewhere belongs to $\M_{\mathbf {s}}(Well(V))$. Hence, according to the induction hypothesis, there is some $x'\in \M$ such that $\psi (x')= f'$.  Let  $B:= \hat B_{\M}(x', r)$ and  $\zeta:= f(r)$. We claim that there is some $a\in B$ such that $\psi(a)(r)=\zeta$. Indeed, since $\zeta\in \mathbf s(r)$, it suffices to observe that  $\tilde \psi(B)= \mathbf s(r)$ (this is immediate: since $\M$ satisfies Property ${\bf h_1}$, $B\in \Nerv(\M)$ and $\delta(B)=r$ and since $\M$ satisfies  Property ${\bf h_2}$, $\mathbf s_{\M}(B)= \mathbf s(r)$. According to Claim \ref{claim:extrastep}, $\tilde \psi(B)= \mathbf s_{\M}(B)$. The result follows). Now, let $\alpha$ be minimum such that $a_{\alpha}\in B$ and $\psi(a_{\alpha})(r)=\zeta$. We claim that $\psi(a_{\alpha})=f$. Indeed, according to Claim \ref {claim:extrastep}, if $r'<r$ and $r'\in V_{*}$ then $\psi(a_{\alpha})(r')=0$. Since $f(r')=0$, we have $\psi(a_{\alpha})(r')=f(r')$. If $r'=r$ we have $\psi(a_{\alpha})(r')=\zeta=f(r')$. Now, if $r<r'$ then $\psi(a_{\alpha})(r')= f'(r')=f(r')$ since $d(\psi(a_{\alpha}), f')=d(a_{\alpha}, x')\leq r$. This proves our claim.  
 \end{proof}

The proof of the equivalence between $(i)$ and $(ii)$ in the result above is due to  Fe\v{\i}nberg  \cite {feinberg1}.  The following result is essentially Proposition 33, p. 416 of  \cite{delon}. 
\begin{theorem}
The ultrametric space $\M_{\mathbf {s}}(Fin(V))$ is embeddable into every ultrametric space $\M$ with property $\mathbf h$ such that $\Spec(\M)\supseteq V$ and $\mathbf {s}_{\M}\geq \mathbf {s}$. 
\end{theorem}
\begin{proof} Let $V':= \Spec(\M)$ and $\mathbf {s'}:= \mathbf {s}_{\M}$. Let $\psi$ be an embedding from $\M$ into $\M_{\mathbf {s'}}(Well(V'))$ given by the proof of Theorem \ref{mainresult}.  It suffices to prove that  $V_{\mathbf {s}}(Fin(V))$ is included into the range of $\psi$ and hence that  for each non-negative integer  $n$, every $f\in  V_{\mathbf {s}}(Well(V))$ with $\vert\supp (f)\vert=n$ there is some $x\in \M$ such that $\psi (x)=f$. The proof is by induction on $n$ and exactly the same as in Case 2 of the proof of Theorem   \ref{thm:feinberg}.  
\end{proof}

Let us recall that the \emph{weight}   of a topological space $X$ is the minimum cardinality $\weight(X)$ of a base (a set $\mathcal B$ of open sets of $X$ such each open set of $X$ is the union of  members of $\mathcal B$), the $\pi$-{weight}  is the minimum cardinality $\piweight(X)$ of a set $\mathcal B$ of non-empty open sets such that every non-empty open set contains some member of $\mathcal B$; the \emph{density} of $X$ is the minimum cardinality $\density(X)$ of a dense set. As it is well known, $\weight(X)=\piweight(X) =\density(X)$ if $X$ is metrizable. In particular, density, $\pi$-weight and weight coincide on ultrametric spaces.

Note that for an  ultrametric $\M$ the following inequality holds:

\begin{equation} \label{density} \Max \{ \vert \Spec (\M)\vert, \Sup \{ \mathbf {s}_{\M}(r): r\in \Spec_{*} (\M) \} \leq \density (\M). 
\end{equation}

%Let $\kappa$ be a  cardinal  at least equal to $2$ and $V$ be a subset of $\RR^{+}$ containing $0$.  Set $V_{\kappa}$ for the set of maps from $V_{*}$ into $\kappa$ and $V_{\kappa}(Well)$ the subset of those with dually well ordered support. Let  $\M_{\kappa}(Well)$ be the ultrametric space on $V_{\kappa}(Well)$. 

From Theorem \ref{mainresult} and inequality (\ref{density}) we get:

\begin{theorem} Every ultrametric space $\M$ with spectrum included into $V$ and density at most $\kappa$  is isometrically embeddable into $\M_{\kappa}(Well)$. 
\end{theorem}
With the fact that $d(\M_{\kappa}(Well))\leq  \kappa^{\aleph_0}$ proved below we get the result of A. and V.  Lemin \cite{lemin2}).

\begin{fact}
Let  $V$ be an infinite subset of $\RR^{+}$ containing $0$. 
\begin{enumerate}
\item $d(\M_{\kappa}(Well))= \vert \M_{\kappa}(Well)\vert = \kappa^{\aleph_0}$ if $V_{*}$ contains a subset of type $1+ \omega^*$
\item $d(\M_{\kappa}(Well))=\Max \{\kappa, \vert V\vert\}$ and  $\vert \M_{\kappa}(Well)\vert = \kappa^{\aleph_0}$ if $V_{*}$ contains a subset of type $\omega^*$ but no subset of type $1+\omega^{*}$.
\item $d(\M_{\kappa}(Well))=\vert \M_{\kappa}(Well)\vert =  \Max \{\kappa, \vert V\vert\}$ if $V_{*}$ is well founded.
\end{enumerate}
\end{fact}

\begin{proof} From inequality (\ref{density}),  we have  $\Max\{\kappa, \vert V\vert\} \leq d(\M_{\kappa}(Well)) \leq \vert V_{\kappa}(Well)\vert$. Since $V_{*}$ is a subset of $\RR$, each  dually well ordered subsets is countable, hence its contribution to $V_{\kappa}(Well)$ has cardinality at most $\kappa^{\aleph_0}$; since the number of these subsets is at most $2^{\aleph_0}$, $\vert V_{\kappa}(Well)\vert \leq \kappa^{\aleph_0}\cdot 2^{\aleph_0}=\kappa^{\aleph_0}$. If $V$ is not well founded, $\kappa^{\aleph_0}$ is attained. If $V$ is well founded, no infinite subset of $V_{*}$ is dually well founded hence $\vert V_{\kappa}(Well)\vert \leq \kappa^{<\omega}\cdot \vert V\vert=Max \{\kappa, \vert V\vert\}$. In this case,  we obtain the equalities in item $3$.  If $V_{*}$ contains a subset of type $\omega^*$ but no subset of type $1+\omega^{*}$ then $d(\M_{\kappa}(Well))=  \piweight(\M_{\kappa}(Well))\leq  \vert \Nerv(\M_{\kappa}(Well))\setminus \M_{\kappa}(Well)\vert \leq \Max\{\kappa, \vert V\vert\}$, hence Item $2$ holds.  If $V_{*}$ contains a subset of type $1+ \omega^*$ then there are $\kappa^{\aleph_0}$ pairwise disjoint members of $\Nerv(\M_{\kappa}(Well))$ with non-empty diameter, hence $d(\M_{\kappa}(Well))\geq \kappa^{\aleph_0}$.  The equality $d(\M_{\kappa}(Well))=\kappa^{\aleph_0}$ follows. This proves that Item $1$ holds. 
\end{proof}

\section{Spectral   homogeneity}

Examples \ref{ex:counter} show that  cardinality conditions are not   sufficient to imply homogeneity of an ultrametric space. In this section we introduce a necessary and sufficient condition for homogeneity (Theorem \ref{cor:char}), from which we derive the fact that  homogeneity is preserved under  Cauchy completion 
(Theorem \ref {thm:cauchy}). These results are by-products of the study of more general ultrametric spaces, that we call spec-homogeneous ultrametric spaces.

\begin{definitions}
Let $\M$ be a metric space. A \emph{local isometry} of\/  $\M$  is an isometry from a metric subspace of\/  $\M$ onto another one. A \emph{local spectral-isometry}, in brief a \emph{local spec-isometry}, is a local isometry $f$ of\/  $\M$  such that:
\begin{equation} \Spec(\M, x)= \Spec(\M, f(x)) \; \text{for every}\;  x\in \dom(f).  
\end{equation}
The space $\M$ is  \emph{spec-homogeneous}  if every local spec-isometry of $\M$ with finite domain extends to a surjective isometry of\/  $\M$.
\end{definitions}

\begin{definitions}\label{defin:similar} 
Two balls $B, B'\in \Ball(\M)$  are \emph{similar} whenever   \begin{enumerate}
\item \label{item:past1} $B$ and $B'$ have the same kind. 
\item \label{item:past2}There are $x\in B$, $x'\in B'$ such that $\Spec(\M, x)=\Spec(\M, x')$. 
\end{enumerate}

Let $X\subseteq M$. The \emph {past} of $X$ in $\M$ is the set:
$$\Past(\M, X):= \{\delta(B): X\subseteq B\in \Nerv(\M)\}.$$
\end{definitions}

\begin{lemma}\label{lem:similar}
Two balls $B, B'$ of an ultrametric space $\M$ are similar if and only if  they have the same past  and the multispectra of $\M_{\restriction B}$ and $\M_{\restriction B'}$ intersect.
\end{lemma}
\begin{proof}
Suppose that $B$ and $B'$ have the same past and the multispectra of $\M_{\restriction B}$ and $\M_{\restriction B'}$ intersect. Item {\em (1)} follows from the fact that  the multispectra of $\M_{\restriction B}$ and $\M_{\restriction B'}$ intersect. Item {\em (2)} follows from Formula (\ref{past}) below: 
\begin{align}\label{past}
\Spec( \M, y)= \Spec(\M_{\restriction B}, y)\cup \Past (\M, B).
\end{align}
for every $B\in \Ball(\M)$ and $y\in B$.  

The converse follows from Formula (\ref{past2})
below:
\begin{equation}\label{past2}
\Spec(\M_{\restriction B}, y)=\Spec(\M, y)\cap X
\end{equation}
where $X:=]0, \delta(B)]$ if $\delta(B)$ is attained and $X:]0, \delta(B)[$ otherwise.
\end{proof}

Here is our next result.
\begin{theorem}\label{thm:char} An ultrametric space is spec-homogeneous if and only if for any pair  of similar balls $B, B'\in \Ball(\M)$, the subspaces $\M_{\restriction B}$ and $\M_{\restriction B'}$ are isometric. 
\end{theorem}
\begin{proof}
We argue first that  the condition on balls is necessary. Let  $B,B'\in \Ball(\M)$ which are similar. According to Definition   \ref{defin:similar} we have $\Spec(\M, x)=\Spec(\M, x')$ for some elements $x\in B$ and $x^\prime\in B^\prime$. Since $\M$ is spec-homogeneous, there is some $\varphi \in \Iso(\M)$ such that $\varphi (x)=x'$. Let $r:=\delta(B)$. Since $\M$ is an ultrametric space, $B=\hat{B}(x,r)$ if $\delta(B)$ is attained and $B=B(x,r)$ otherwise. Since $B$ and $B'$ are similar,  $\delta(B')=r$  and $B'=\hat{B}(x',r)$ if $\delta(B')$ is attained, $B'=B(x',r)$ otherwise. Clearly, $\varphi $ is an isometry from $\M_{\restriction B}$ onto $\M_{\restriction B'}$. This proves that the condition on balls is necessary. 

Next we prove that  the condition on balls suffices. According to Theorem \ref{thm:first}, this amounts to prove that for every $x,x'\in M$ with the same spectrum there is an automorphism $\varphi$  which carries $x$ onto $x'$.  Let $x, x'\in M$ such that $\Spec(\M, x)=\Spec (\M, x'):=S$. For $r\in S$,  set $D(x, r):= \hat{B}(x, r)\setminus B(x, r)$ and similarly $D(x',  r):= \hat{B}(x, r)\setminus B(x, r)$. The existence of $\varphi$ amounts to the existence, 
for each $r\in S$, of an   isometry $\varphi _r$ from $\M_{\restriction D(x, r)}$ onto $\M_{\restriction D(x',r)}$ (since,   as it is easy to check, the map  which send $x$ to $x'$ and coincides with $\bigcup_{r\in S}\varphi_r$ on $\bigcup_{r\in S}D(x, r)$ is an isometry from $\M$ into itself, this condition suffices; the converse is obvious).  Again, the existence of $\varphi_r$ amounts to the existence of a bijective map $\psi: \Sons(\hat{B}(x,r))\rightarrow \Sons(\hat{B}(x', r))$ such that:

\begin{equation} \label{eq:son1}\psi(B(x,r))=B(x',r)
\end{equation}
\noindent and 
\begin{equation}\label{eq:son2}
\M_{\restriction D}\; \text{is isometric to}\;  \M_{\restriction \psi (D)}
\end{equation} for every $D\in \Sons(\hat{B}(x, r))\setminus \{B(x,r)\}$.
    Indeed, if for each $D\in \Sons(\hat{B}(x, r))\setminus \{B(x,r)\}\}$ there is an isometry $\varphi_D$  of $\M_D$ onto  $\M_{\psi (D)}$, the map $\varphi_r:= \bigcup\{\varphi_D: D\in \Sons(\hat{B}(x, r))\setminus \{B(x,r)\}\}$,  is an isometry of $\M_{\restriction D(x, r)}$ onto $\M_{\restriction D(x',r)}$. The converse is obvious.  In our case,  the balls  $\hat{B}(x, r)$ and $\hat{B}(x', r)$ being similar, by assumption there is an isometry $\varphi$ of $\M_{\restriction \hat{B}(x,r)}$ onto $\M_{\restriction \hat{B}(x',r)}$. This isometry induces a bijective map $\varphi_{*}$ from $\mathcal B:= \Sons(\hat{B}(x, r))$ onto $\mathcal B':= \Sons(\hat{B}(x', r))$ such that  $\M_{\restriction D}$ is isometric to $\M_{\restriction \varphi_{*} (D)}$ for every $D\in \mathcal B$. But, since there is no reason for which $\varphi_{*}(B(x,r))=B(x',r)$, some modification of $\varphi_{*}$  is needed. Let $\mathcal A$ be the set of $D\in \mathcal B$ such that $\M_{\restriction D}$ is isomorphic to $\M_{\restriction B(x,r)}$ and  let $\mathcal A'$ be the subset of $\mathcal B'$ defined in a similar way. Since $x$ and $x'$ have the same spectrum,  $B(x,r)$ and $B(x',r)$ are similar, hence, according to our hypotheses, $\M_{\restriction B(x,r)}$  and $\M_{\restriction B(x',r)}$ are isometric. Since $B(x,r)$ and $B(x',r)$ belong to $\mathcal A$ and $\mathcal A'$ respectively, $\varphi_{*}$ induces a bijection of $\mathcal A$ onto $\mathcal A'$. Set $D':= \varphi_{*}(B(x,r))$ and $D'':=\varphi^{-1}_{*}(B(x',r))$ and replace  $\varphi_{*}$ by the map $\psi$ which send $B(x,r)$ to $B(x',r)$, $D''$ to $D'$ and  coincides with $\varphi_{*}$ on the other members of $\mathcal B$. The map $\psi$ satisfies Conditions (\ref{eq:son1}) and (\ref{eq:son2}) above. It yields the required $\varphi_r$.    
\end{proof}
As a corollary, we obtain:

\begin{theorem}\label{cor:char}
An ultrametric space $\M$ is homogeneous if and only if 
\begin{enumerate}
\item $\Spec(\M,x)=\Spec(\M,x^\prime)$ for all $x,x^\prime\in M$. 
\item Balls  of $\M$ of the same kind are isometric.
\end{enumerate}
\end{theorem}

From Theorem \ref{cor:char} we deduce: 
\begin{theorem}\label{thm:cauchy} The Cauchy completion of an homogeneous ultrametric space is homogeneous. 
\end{theorem}

\begin{proof} Let $\M:= (M, d)$ be a homogeneous ultrametric space and $\overline \M:=(\overline M, \overline d)$ be its Cauchy completion. We show that $\overline \M$ satisfies the two conditions in Theorem  \ref{cor:char}. We may suppose that $M$ is a dense subset of $\overline M$ and that $d$ is the restriction of $\overline d$ to $M$. 
Since $M$ is dense in $\overline \M$, $\Spec(\overline \M)=\Spec(\M)$, furthermore $\Spec(\overline \M, x)=\Spec(\overline \M)$ for every $x\in \overline M$ (Lemma \ref{lem:dense}) hence Condition (1) of Theorem \ref{cor:char} is satisfied. 
 Let $B_0$ and $B_1$ be two  balls in $\overline \M$ having the same kind. We may suppose that these two balls are non-trivial.   Set $B'_i:=B_i\cap M$ for $i<2$. According to Lemma \ref{lem:dense}, $B'_i$ and $B_i$ have the same kind. Thus $B'_{0}$ and $B'_{1}$ have the same kind. Since $\M$ is homogeneous, $\M_{\restriction B'_{0}}$ and  $\M_{\restriction B'_{1}}$ are isometric. An isometry from $\M_{\restriction B'_{0}}$ onto $\M_{\restriction B'_{1}}$ extends uniquely to the Cauchy completions of these spaces.  Since by Lemma \ref{lem:dense}, $B'_i$ is dense in $B_i$ and $B_i$ is topologically closed in $\overline \M$ (no matter the kind of $B_i$), $\M_{\restriction B_{i}}$ is  the completion of $\M_{\restriction B'_{i}}$. Hence the two balls $B_0$ and $B_1$ are  isometric and Condition (2) of Theorem \ref{cor:char} is satisfied. 
\end{proof}
\subsection{Countable spec-homogeneous ultrametric spaces}
Theorem \ref{thm:char} leads to the following problem:
\begin{problem}\label{prob:ultra} Is  an ultrametric space $\M$ spec-homogeneous whenever for any pair  of similar balls $B, B'\in \Nerv(\M)$, the subspaces $\M_{\restriction B}$ and $\M_{\restriction B'}$ are isometric?
\end{problem}

According to the next theorem below, a counterexample to Problem \ref {prob:ultra} must be uncountable. But first for convenience we label the desired property.

\begin{definition} An ultrametric space $\M$ satisfies condition (A) if 
for every $B,B'\in \Nerv(\M)$ the subspaces $\M_{\restriction B}$ and $\M_{\restriction B'}$ are isometric provided that $B$ and $B'$ are similar. It satisfies  condition (B) if 
for every open balls $B:=B(x,r),B':= B(x',r)$ the subspaces $\M_{\restriction B}$ and $\M_{\restriction B'}$ have the same multispectrum   provided that $B$ and $B'$ are similar.
\end{definition}

\begin{theorem}\label{thm:main}
A countable ultrametric space  is spec-homogeneous if and only if it satisfies condition (A).
\end{theorem}

The key ingredients of the proof of this result are the following notion and Lemma \ref{lem:conditionB} below. 

\begin{definition} A metric space $\M$ satisfies the \emph{finite spec-extension property} if every local-spec isometry with finite domain extends at any new element to a local spec-isometry.
\end{definition}

We note that a spec-homogeneous metric space satisfies the finite spec-extension property. The converse holds if $\M$ is countable. This noticeable fact holds for larger classes of structures. The proof uses the back and forth  method invented by Cantor (see \cite{fraisse}, \cite{hodges}).

\begin{lemma} \label{lem:conditionB}Every ultrametric space satisfying  condition (A) satisfies condition (B).
\end{lemma}

\begin{proof}
Let $B:=B(x,r),B':= B(x',r)$ be similar balls, and thus by   Lemma \ref{lem:similar} they have the same past and the  multispectra of $\M_{\restriction B}$ and $\M_{\restriction B'}$ intersect. Let  $z\in B$. We claim that there is some $z'\in B'$ with the same spectrum. We observe  first that according to  Definition \ref{defin:similar},    $B$ and $B'$  contain $x_1$ and $x'_1$ such that $\Spec(\M,x_1)=\Spec (\M, x'_1)$. 
Set $s:=d(z, x_1)$, $B_1:= \hat{B}(x_1, s)$ and $B'_1:= \hat{B}(x'_1, s)$.  We have $s\in \Spec(\M, x_1)$  and,  since  $\Spec(\M,x_1)=\Spec (\M, x'_1)$,  $s\in \Spec(\M, x'_1)$, proving that $B_1, B'_1\in \Nerv(\M)$. Clearly $B_1$ and $B'_1$ have the same past; hence, according to Condition (A), $\M_{\restriction B_1}$ and $\M_{\restriction B'_1}$ are isometric. It follows that   $B'_1$ contains some element $z'$  such that  $\Spec(\M_{\restriction B_1}, z)=\Spec (\M_{\restriction B'_1}, z')$. According to Lemma \ref{lem:similar}, $\Spec(\M, z')=\Spec (\M, z)$. This  proves our claim. We get the same conclusion if we exchange the roles of $B$ and $B'$. This suffices to prove the lemma.\end{proof}

\subsection{Proof of Theorem \ref{thm:main}}\begin{proof}
Let $\M:= (M,d)$ satisfying Condition (A). Since $M$ is countable, it suffices to  prove that $\M$ satisfies the finite spec-extension property. Let $\varphi: F\to M$ be a local spec-isometry of $\M$ and let $a\in M\setminus F$.
If $F$ is empty, then $\varphi$ extends to $\{a\}$ by the identity.
Assume that $F$ is non-empty, then set $r:=d(a, F)$,  $B:=\hat{B}(a,r)$ and  $A:=F\cap B$.
Our goal is to find an element of $M$ at distance $r$ from every element of $\varphi[A]$
and with the same spectrum as $a$. We have $B\in \Nerv(\M)$, $\delta(B)= r$ and furthermore $B=\hat{B}(a',r)$ for every   $a' \in A$. Since $\Spec(\M, a')= \Spec(\M, \varphi(a'))$ for every $a' \in A$,  the closed ball $B':=  B(\varphi(a'),r)$ belongs to $\Nerv(\M)$ and is independent from the choice of $a'$ in $A$. 
We have to find a son of $B'$ disjoint from $\varphi[A]$ whose multispectrum contains the spectrum  of $a$. Since $\varphi$ preserves spectra, $B$ and $B'$ have the same past. 
So by assumption $\M_{\restriction B}$ and $\M_{\restriction B'}$ are isometric. Let  $\psi:B\to B'$ be an isometry.
Then  $\psi$ induces a bijection $\psi_*$ from 
$\mathcal B:=\Sons(B)$ onto $\mathcal B':=\Sons(B')$.
In particular  $\M_{\restriction D}$ and $\M_{\restriction \psi_*(D)}$ have the same multispectrum for every $D\in \mathcal B$.

Let $\mathcal A$ denote the set of sons of $B$ that meet $A$. Then 
$\varphi$ yields an injection $\varphi_*$ from $\mathcal A$ into $\mathcal B'$.
Furthermore, for each $D\in \mathcal A$, $D$ and $\varphi_*(D)$ are two open balls of radius $r$ having the same past and such that 
the multispectra of $\M_{\restriction D}$ and $\M_{\restriction \varphi_*(D)}$ meet. Therefore, according to Lemma \ref {lem:conditionB} above, $\M_{\restriction D}$ and $\M_{\restriction \varphi_*(D)}$ have the same multispectrum.

Let  $C$ be the son of  $B$ containing $a$. 
Let $\mathcal A^+$ denote the set of members $D$ of $\mathcal A$ such that $\Spec(\M_{\restriction D}, d)= \Spec(\M_{\restriction C}, a)$ for some $d\in D$, and let $\mathcal A^-:=\mathcal A\setminus \mathcal A^+$.
We claim that:\[
\psi_*[\mathcal A^+ \cup\{C\}]\not\subseteq\varphi_*[\mathcal A]
\]
Indeed, first $\psi_*[\mathcal A^+ \cup\{C\}]$ is disjoint from 
$\varphi_*[\mathcal A^-]$ since the multispectrum of any member of $\psi_*[\mathcal A^+ \cup\{C\}]$ contains $\Spec(\M_{\restriction C}, a)$ while none of those of $\varphi_*[\mathcal A^-]$ does. Next, since $C$ does not meet $F$ and the maps $\psi_*$ and $\varphi$ are  one-to-one, the size of $\psi_*[\mathcal A^+ \cup\{C\}]$ is larger than the size  of 
$\varphi_*[\mathcal A^+]$.  This proves that  $\psi_*[\mathcal A^+ \cup\{C\}]$  is not included in $\varphi_*[\mathcal A]=\varphi_*[\mathcal A^+]\cup\varphi_*[\mathcal A^-]$, as claimed.

To conclude the proof of the theorem, observe that  any member of $\psi_*[\mathcal A^+ \cup\{C\}]\setminus\varphi_*[\mathcal A]$ contains an element  $a'$ with the same spectrum as $a$ and at distance $r$ from any element of $\varphi[A]$.
\end{proof}

\begin{corollary}\label{cor:main}
A countable ultrametric space  is spec-homogeneous provided that any two members of the nerve with the same diameter are isometric.
\end{corollary}

A construction of ultrametric spaces satisfying this condition is given in \cite{DLPS2}. To each ultrametric space $\M$ we associate an ultrametric space, the \emph{path extension}   of  $\M$, denoted by  $\mathbb \Path(\M)$,   whose elements, the \emph {paths},  are special finite unions of chains in $(\Nerv(\M),  \supseteq)$.

We recall that if  $\M:=(M,d)$ is an ultrametric space and   $\alpha\in \RR^*_{+}\cup \{+\infty\}$, the pair $\M_{\alpha}:= (M, d\wedge {\alpha})$, where $d\wedge{\alpha}(x,y):=\min(\{d(x,y), \alpha \})$, is an ultrametric space.
We recall the following properties of the path extension (see Theorem 9 of \cite{DLPS2}).

\begin{proposition}\label{thm:pathext2}
For every ultrametric space $\M$, the path extension  $\mathbb \Path (\M)$ of $\M$ satisfies the following properties:
\begin{enumerate}
\item  \label{item:2}$\mathbb \Path (\M)$ is an isometric extension of $\M$ with the same spectrum as $\M$ and $\vert  \Path (\M)\vert\leq \vert M\vert+ \aleph_{0}$.
\item\label{item:6} $\mathbb \Path (\M)_{\restriction B}$ is  isometric to $\mathbb \Path (\M_{\delta(B)})$ for every non-trivial \\$B\in \Nerv(\mathbb \Path (\M))$.
\end{enumerate}
\end{proposition}

With Corollary \ref{cor:main} we obtain our final result:
\begin{theorem}\label{cor:pathext2}
Every countable ultrametric space extends via a spec-isometry to a countable spec-homogeneous ultrametric space.\end{theorem}
\begin{proof}
Let  $\M$ be an ultrametric space. Set $\M':=\mathbb \Path (\M)$. It follows from  Item \ref{item:6} of Proposition \ref{thm:pathext2} that $\M'_{\restriction B}$ and $\M'_{\restriction B'}$ are isometric provided that $B, B'\in \Nerv(\M')$ and $\delta(B)=\delta(B')$. Thus, according to Corollary \ref{cor:main}, $\M'$ is spec-homogeneous provided that it is countable. According to Item {\em (1)} of Proposition \ref{thm:pathext2},  this is the case if $\M$ is countable. To conclude it suffices to observe that the natural isometry $i:\M\rightarrow \M'$ is in fact a spec-isometry, that is satisfies $\Spec(\M, x)=\Spec(\M', x)$ for all $x\in M$. For that,  note that if $B\in \Nerv(\M')$  and $\alpha= \delta(B)$ then $\Spec(\M'_{\restriction B})=\Spec(\Path (\M_{\alpha}))=\Spec(\M_{\alpha})$.  (The first equality follows from  Item {\em (2)} and the second from Item {\em (1)}.) \end{proof}

%\item The map $\varphi_{\alpha}$ is an isometric embedding of  $\M_{\alpha}$ into $\mathbb \Path_{\alpha}(\M)$. Moreover $spec(\M_{\alpha}, x)=spec (\mathbb \Path_{\alpha} (\M), \varphi_{\alpha}(x))$ for every $x\in \M$.
%\item $\Spec(\M_{\alpha})=\Spec(\mathbb \Path_{\alpha}(\M))$.
% 

%Let $\mathcal I\in \Path_{\alpha}(\M)$ such that $d(\varphi_{\alpha}(x), \mathcal I)= r$. Hence for $Y:= \mathcal{I}\wedge_{\alpha} \varphi_{\alpha}(x)$ we have $\delta (Y)=r$.  Pick $y\in \M$ such that $\{y\}\in \mathcal I$. We have $d(x,y)\geq  \delta (Y)$. Thus $d(x,y)\geq \alpha$ if $\delta(Y)=\alpha$. If not,-  there is $B\in \Nerv(M)$ such that $Y=]_{\alpha}\leftarrow B]$. In this case $\delta(B)=r$. Thus  $r\in \Spec(\M_{\alpha}, x)$ .    

\section{Ultrametric spaces and $2$-structures}\label{section:modules}

%On a m\'elang\'e 
%trees et forests, et on ne les a pas orient\'es racine en bas.
%Rectifier~!

Let us  recall some basic ingredients of the theory of $2$-structures (see \cite {ehrenfeucht}). Let $V$ be  a set. A  $V$-\emph{labelled   $2$-structure} (briefly, a $2$-structure) is a pair $(E, v)$ where $v$ is a map from $E\times E\backslash\Delta_E$ into  $V$,
with $\Delta_E:=\{(x,x):x\in E\}$ denoting the diagonal of $E$.

A subset $A$ of  $E$ is a \emph{module} if $v(x,y)=v(x,y')$  and $v(y,x)=v(y',x)$ for every $x\in E\setminus A$ and  $y,y'\in A$
(other names are autonomous sets, or intervals). 

The whole set, the empty set  and the 
singletons are modules. These are the  \emph{trivial} modules. We recall  first the basic and well-known properties of modules. %, under the form given in \cite{courcelle}. 
\begin{lemma}
Let $(E,v)$ be  a $2$-structure.
\begin{enumerate}
\item The intersection of a non-empty set of modules is a module (with $\cap \emptyset = E$). 
\item  The union of two modules that meet is a module, and more generally, the union of a set of modules is a module as
soon as the meeting relation on that set is connected. 
\item For two modules A and B, if $A\setminus B$ is non-empty, then $B\setminus A$ is a module. 
\end{enumerate}
\end{lemma}

A module is \emph{strong} if it  is comparable (\wrt\ inclusion) to  every  module it meets. 

\begin{lemma} \label{strongmodule}
Let $(E,v)$ be  a $2$-structure.
\begin{enumerate} 
\item The intersection of any set of strong modules is a strong module. 
\item The union of any directed set of strong modules is a strong module.
\end{enumerate}
\end{lemma}

A module is \emph{robust} if this is the least strong  module containing two elements of $V$.  
%With this property, it turns out that robust modules exist. 
The robust modules of a binary structure $(E, v)$ form a tree under reverse inclusion.
This tree is canonically endowed with an additional structure from which the given $2$-structure can be recovered (see \cite{courcelle}).
First we call a $2$-structure $(E, v)$ is {\em symmetric} if $v(y,x)=v(x,y)$ for every distinct $x$ and $y$ in $E$. Further, we call a $2$-structure $(E, v)$ is {\em hereditary decomposable} if every induced substructure on at least three elements of $E$ has a nontrivial module.

Now an ultrametric space $\M:=(M, d)$ is  obviously a symmetric $2$-structure with values in the set of positive reals. It is also hereditary decomposable since every ultrametric space with at least three elements has a nontrivial module. Indeed, if the distance assumes only one nonzero value then consider any pair of two distinct elements of $M$, and if it assume at least two values $r<s$ then consider the close ball $\hat B(x,r)$ for any $x\in E$ such that there is a $y\in E$ with $d(x,y)=r$.

Thus since our interest lies in ultrametric spaces, we simply consider symmetric hereditary decomposable $2$-structures. This additional structure on the tree of robust modules is just a labelling  of the nonsingleton nodes into $V$; indeed to each nonsingleton robust module $R$ there corresponds some $v(R)\in V$ such that for any two distinct $x$ and $y$ in $E$, $v(x,y)=v(R)$, where $R$ is the least strong module containing $x$ and $y$. Let us call  the \emph{decomposition tree of a symmetric hereditary decomposable $2$-structure} its tree of robust modules endowed with this node labelling into $V$.

%Their being modules is handled by Corollary~\ref{strong, robust} below.

\begin{proposition}\label{prop:nerve}
The decomposition tree of an ultrametric space $\M$ is equal to its nerve endowed with the diameter function. 
\end{proposition}

This is the last statement of Corollary~\ref{strong, robust}. 

In the statement below,  we allow balls of  infinite radius (which are equal to $M$).

\begin{lemma} Let  $\M:=(M, d)$ be an  ultrametric space. The least module including  a subset $A$ of $M$ is the union of all open balls of radius $\delta (A)$  centered in $A$ (this is the set $\bigcup_{a\in A} B(a, \delta (A))$). In particular, if $A$ is unbounded, this module is $M$.  As a consequence, $A$ is a module \iff\ it is a union of open balls of radius $\delta(A)$. 
\end{lemma}

\begin{proof}
Given $A\subseteq M$, let $A'$ denote 
the union of all open balls of radius $\delta (A)$  centered in $A$.
We first check that $A'$ is included in any module $C$ of $\M$ that includes $A$.
Indeed given any $a'\in A'$, 
consider some $a\in A$ such that $d(a,a')<\delta(A)$.
Then, recalling that $\delta(A)=\sup\{d(a,b):b\in A\}$,
consider some $b\in A$ such that $d(a,a')<d(a,b)$. 
It follows from the strong triangle inequality that $d(a,a')<d(a,b)=d(a',b)$.
Thus $a'$, which distinguishes the two elements $a$ and $b$ of $C$, belongs to $C$.
Then we check that $A'$ is a module.
Consider $x\in M\backslash A'$ and $a'$ and $b'$ in $A'$.
Then consider $a$ and $b$ in $A$ such that $d(a,a')<\delta(A)$
and $d(b,b')<\delta(A)$.
Thus on the one hand $d(x,a)\geq\delta(A)\geq d(a,b)$
from which the strong triangle inequality yields 
$d(x,a)=d(x,b)$.
On the other hand  $d(x,a)\geq\delta(A)\geq d(a,a')$, hence $d(x,a)=d(x,a')$.
Likewise $d(x,b)=d(x,b')$.
Finally $d(x,a')=d(x,a)=d(x,b)=d(x,b')$.
\end{proof}

\begin{corollary}\label{strong, robust}
Let $\M:=(M,d)$ be an ultrametric space.
\begin{enumerate}
\item The strong modules  of $\M$ are the balls (open or closed) of $\M$,   that is the sets $B(a,r)$ and   $\hat{B}(a, r)$ for some $a$ and $r$, and $M$. 
\item The robust modules are the closed balls attaining their diameter. These are  the sets  of the form $\hat {B}(a, d(a, b))$  for some $a, b\in M$. 
\end{enumerate}
\end{corollary}

\begin{proof}
\begin{enumerate}
\item 
Let us first check that any ball $B$ of attained diameter $r$ is a strong module.
It follows from the lemma that this is a module since $B=\cup_{a\in B}B(a,r)$.
Now given a module $C$ meeting $B$, say with $a\in B\cap C$,
so that $B=\hat B(a,r)$,
let us check that $B$ and $C$ are comparable.
If $\delta(C)\leq r$ then $C\subseteq B$.
If $\delta(C)>r$, then $\hat B(a,r)\subseteq B(a,\delta(C))$ which is included in $C$,
according to the lemma.
Thus $B$ is a strong module.
Finally observe that any open ball $B(a,r)$, as a union of $\cup_{s<r}\hat B(a,s)$ of a chain of strong modules, is a strong module itself.
\item
Consider a nonempty ball $B$.
If it attains its diameter $r$, then given any $a$ and $b$ in $B$ with $d(a,b)=r$,
$B$ is the least ball containing $a$ and $b$, hence the least strong module containing $a$ and $b$. In this case $B$ is robust.
If it does not attain its diameter $r\leq\infty$, then for any $a$ and $b$ in $B$,
the least ball containing $a$ and $b$ has diameter $d(a,b)<r$ and therefore is distinct from $B$.
In this case $B$ cannot be robust.
\end{enumerate}
\end{proof}

A $2$-structure is said to be {\em strong-modular complete} if every nonempty chain of strong modules has a nonempty intersection.
According to Corollary~\ref{strong, robust}, we conclude that the strong-modular complete ultrametric spaces are the T-complete ultrametric spaces.

\end{document}